\documentclass[11pt,dvipsnames]{amsart}
\usepackage{amsmath,amssymb,amsfonts, mathtools}
\usepackage[mathscr]{eucal}
\usepackage{fullpage}
\usepackage{quiver}

\usepackage{tikz, tikz-cd}
\usetikzlibrary{matrix,arrows,decorations.pathmorphing,cd,patterns,calc,backgrounds,patterns}
\tikzset{dummy/.style= {circle,fill,draw,inner sep=0pt,minimum size=1.2mm}}
\tikzset{vertex/.style={fill, circle, minimum size=.1cm, inner sep=0pt}}

\usepackage[pagebackref, colorlinks, citecolor=DarkOrchid, linkcolor=NavyBlue, urlcolor=NavyBlue]{hyperref}
\hypersetup{linktoc=all,} 

\usepackage{comment}
\usepackage{hyperref}

\usepackage[textwidth=25mm, textsize=tiny]{todonotes}
\usepackage[final]{microtype}

\usepackage{enumerate}

\usepackage{appendix}

\numberwithin{equation}{section} 
\numberwithin{figure}{section}
\usepackage[nameinlink,capitalise,noabbrev]{cleveref}

\newcommand{\newrefformat}[2]{}

\usepackage{stmaryrd}

\newcommand\restr[2]{{
  \left.\kern-\nulldelimiterspace 
  #1 
  \vphantom{\big|} 
  \right|_{#2} 
  }}

\crefname{lemma}{Lemma}{Lemmas}
\crefname{theorem}{Theorem}{Theorems}
\crefname{definition}{Definition}{Definitions}
\crefname{proposition}{Proposition}{Propositions}
\crefname{remark}{Remark}{Remarks}
\crefname{corollary}{Corollary}{Corollaries}
\crefname{equation}{Equation}{Equations}
\crefname{construction}{Construction}{Constructions}
\crefname{ex}{Example}{Examples}
\crefname{appsec}{Appendix}{Appendices}
\crefname{subsection}{Subsection}{Subsections}

\theoremstyle{plain}
\newtheorem{theorem}[equation]{Theorem}
\newtheorem{corollary}[equation]{Corollary}
\newtheorem{proposition}[equation]{Proposition}
\newtheorem{lemma}[equation]{Lemma}

\newtheorem{introtheorem}{Theorem}
 
\crefname{introtheorem}{Theorem}{Theorems}

\theoremstyle{definition}
\newtheorem{definition}[equation]{Definition}
\newtheorem{example}[equation]{Example}
\newtheorem{remark}[equation]{Remark}

\newtheorem*{notation}{Notation}

\newcommand{\bS}{\mathbb{S}}

\newcommand{\Map}{\mathrm{Map}}

\newcommand{\cat}[1]{\mathscr{#1}}

\DeclareMathOperator{\Cat}{\mathscr{C}\textup{at}}
\DeclareMathOperator{\Perm}{\mathscr{P}\textup{erm}\mathscr{C}\textup{at}}
\DeclareMathOperator{\PC}{PC}
\DeclareMathOperator{\Sym}{\mathscr{S}\textup{ym}\mathscr{M}\textup{on}\mathscr{C}\textup{at}}
\newcommand{\smc}{\mathscr{S}\textup{ym}\mathscr{M}\textup{on}\mathscr{C}\textup{at}_\infty}
\newcommand{\Mack}{\operatorname{Mack}}
\newcommand{\Sp}{{\rm Sp}}

\newcommand{\op}{\mathrm{op}}

\newcommand{\ob}{\operatorname{Ob}}

\newcommand{\id}{\operatorname{id}}

\newcommand{\Fun}{\operatorname{Fun}}
\newcommand{\ps}{{\rm ps}}

\newcommand{\GMMO}{\mathrm{GMMO}}
\newcommand{\SSS}{\mathbb{S}}
\newcommand{\lvl}{\mathrm{lvl}}

\renewcommand{\AA}{\mathbb{A}}
\newcommand{\BB}{\mathbb{B}}
\newcommand{\AAA}{\widehat{\AA}}
\newcommand{\BBB}{\widehat{\BB}}

\newcommand{\cC}{\cat C}
\newcommand{\cD}{\cat D}
\newcommand{\cM}{\cat M}
\newcommand{\cR}{\cat R}
\newcommand{\cO}{\cat O}
\usepackage[textwidth=25mm, textsize=tiny]{todonotes}

\begin{document}
\begin{abstract}
We provide a unifying approach to different constructions of the algebraic $K$-theory of equivariant symmetric monoidal categories. A consequence of our work is that the $K$-theory functor of Bohmann--Osorno models all connective genuine $G$-spectra.
\end{abstract}

\author[M.E.\ Calle]{Maxine E.\ Calle}
\address{Department of Mathematics, University of Pennsylvania,
209 South 33rd Street,
Philadelphia, PA, 19104, USA}
\email{callem@sas.upenn.edu}

\author[D.\ Chan]{David Chan}
\address{Department of Mathematics, Michigan State University, 619 Red Cedar Road, East Lansing, MI 48824, USA}
\email{chandav2@msu.edu}

\author[M.\ P\'eroux]{Maximilien P\'eroux}
\address{Department of Mathematics, Michigan State University, 619 Red Cedar Road, East Lansing, MI 48824, USA}
\email{peroux@msu.edu}

\title[Equivariant $K$-theory of symmetric monoidal Mackey functors]{Equivariant Algebraic $K$-theory\\of symmetric monoidal Mackey functors}


\renewcommand{\subjclassname}{\textup{2020} Mathematics Subject Classification}
\maketitle


\section{Introduction}
Inspired by Quillen's work on the higher algebraic $K$-theory of rings, Segal \cite{segal74} defined a way to produce spectra via the $K$-theory of any symmetric monoidal category using a homotopy theoretic ``group completion'' construction. Since spectra are like homotopical versions of abelian groups, his construction $\cat C\mapsto K(\cat C)$ produces a connective spectrum. The corresponding infinite loop space structure on $K(\cat C)$ explicitly comes from the symmetric monoidal structure on $\cat C$. This $K$-theory construction gives us a way to produce spectra from symmetric monoidal categories, and a natural question to ask is whether all spectra can be modeled this way.

In the 1990s, Thomason \cite{Thomason95} showed Segal's construction gives rise to all connective spectra --- spectra whose negative homotopy groups are trivial. In other words, for every connective spectrum $E$ there is some symmetric monoidal category $\cat C$ so that $K(\cat C)\simeq E$. Thomason proves this result by producing an explicit inverse. A more modern proof of this theorem is given by Mandell in \cite{Mandell/invKT}, and a recent result of Ramzi--Sosnilo--Winges extends this result to show that every spectrum can be modeled by the non-connective $K$-theory of a stable $\infty$-category \cite{RamziSosniloWinges}.

Thomason's and Mandell's work provides a model for the connective stable homotopy category, which can be useful for computations or constructions. Indeed, a large part of Thomason's motivation was to offer a solution to a long-standing problem in stable homotopy theory related to multiplicative structures. The stable homotopy category comes equipped with a well-behaved smash product, but there were technical obstructions to lifting this multiplicative structure to a category of spectra. In fact, Lewis showed there is no such good category with reasonable requirements \cite{lewis}. Thomason's coordinate-free approach using symmetric monoidal categories and $K$-theory was one of the earlier successful attempts to address this issue.
Although Thomason's formulation of the smash product via symmetric monoidal categories has been subsumed by developments in modern stable homotopy theory \cite{EKMM, SS, MMSS, MM,HA}, his work provides a blueprint for modeling generalizations of spectra that are not as well understood.

In the context of equivariant stable homotopy theory, we study an equivariant generalization of spectra called \textit{genuine $G$-spectra}.  While there are several models for genuine $G$-spectra, one widely-used approach is via \textit{Mackey functors}. Essentially, a $G$-Mackey functor in a category $\cat D$ is a collection of objects $\{D(H)\}_{H\leq G}$ in $\cat D$ which are connected by certain $\cat D$-morphisms called restrictions, transfers, and conjugations. Mackey functors should be thought of as the equivariant version of monoid objects in $\cat D$; for more details see \cref{sec:Mack}. Taking $\cat D$ to be a model category of spectra, we obtain a notion of \textit{spectral Mackey functor} which models genuine $G$-spectra. This idea was formalized in the context of enriched category theory in work of Guillou--May \cite{GuillouMay:2011}. An $\infty$-categorical approach was developed by Barwick and coauthors Glasman--Shah \cite{BarwickOne, Barwick2} (see also \cite[Appendix A]{CMNN}). 

One source of genuine $G$-spectra comes from an equivariant version of the $K$-theory of symmetric monoidal categories. Both frameworks for spectral Mackey functors mentioned above are amenable to a version of equivariant algebraic $K$-theory, one due to Bohmann--Osorno \cite{BohmannOsorno:2015} and the other to Barwick \cite{BarwickOne,Barwick2}. These two constructions are similar in flavor but vary in that they are formulated in the contexts of enriched categories and $\infty$-categories, respectively, and hence are built for different notions of \textit{equivariant} symmetric monoidal categories. Other authors have used similar techniques to produce genuine $G$-spectra via different $K$-theory constructions, see \cite{malkiewich/merling:2016, CalleChanMejia:Linearization,Yau:multiplicativeKtheory,GMMO-final}.

Just as genuine $G$-spectra can be modeled as spectral Mackey functors, we would expect the correct notion of equivariant symmetric monoidal structure to be some version of a Mackey functor valued in symmetric monoidal categories. 
The $K$-theory of such an object should then be a levelwise application of non-equivariant $K$-theory. Indeed, this is effectively the approach in both the enriched and $\infty$-categorical setting.  

In this paper, we define the $K$-theory of \textit{symmetric monoidal Mackey functors}, as considered by Hill--Hopkins in \cite{HillHopkins}, which can be thought of as a collection of symmetric monoidal categories whose restrictions and transfers are strictly unital strong monoidal functors.  Our main result is that every connective genuine $G$-spectrum can be modeled up to weak equivalence by the $K$-theory of a such a symmetric monoidal Mackey functor via a functor $\mathbb{K}$ that is morally the same as the construction in \cite{BohmannOsorno:2015}.


\begin{introtheorem}[\cref{cor:BO is all}]\label{intro thomason result}
    For any connective genuine $G$-spectrum $E$ there is a symmetric monoidal Mackey functor $F$ so that $\mathbb{K}(F)$ and $E$ are weakly equivalent as genuine $G$-spectra.
\end{introtheorem}

The functor $\mathbb{K}$ above is adapted slightly from the one in \cite{BohmannOsorno:2015} to resolve some technical complications that arise from working in the enriched setting. In particular, as discussed in \cite[Remark 10.5.5]{johnsonYau:homotopyTheoryOfMackey}, it is unclear whether Bohmann--Osorno's original construction lands in a category of spectral Mackey functors that is Quillen equivalent to the category of genuine orthogonal $G$-spectra. As discussed in \cref{sec:main}, this is due to some subtleties around the compatibility between taking opposite categories and change of enrichment. In \cref{thm:BO}, we show that Bohmann--Osorno's construction can be adapted to produce genuine $G$-spectra; this follows from a covariant version of the Guillou--May theorem \cite{GuillouMay:2011}, proved in \cref{app:GM thm}.

\begin{introtheorem}[\cref{thm:GM no op}]
    There is a zig-zag of Quillen equivalences between the model categories $\Sp^G$ with the stable model structure and the enriched functor category $\Fun_{\Sp}(\mathcal{A}^G_{\Sp}, \Sp)$ with the projective level model structure, where $\mathcal{A}^G_{\Sp}$ is the spectrally enriched Burnside category of \cite{GuillouMay:2011}. 
\end{introtheorem}

An advantage of the functor $\mathbb{K}$ is that it can be used to construct very explicit examples of genuine $G$-spectra (see e.g.\ \cref{ex:G-BPQ} and \cref{ex:ordinary Mackey functor}), and the input of symmetric monoidal Mackey functors are easier to construct than the strict categorical Mackey functors of Bohmann--Osorno. In fact, as we later observe, several previous constructions of equivariant algebraic $K$-theory that use Bohmann--Osorno's machinery, such as \cite{CalleChanMejia:Linearization, malkiewich/merling:2016}, can also be obtained via our construction (see \cref{ex:coefficient ring K theory}). Moreover, every symmetric monoidal Mackey functor can be strictified to produce a (covariant) categorical Mackey functor in the sense of \cite{BohmannOsorno:2015} (\cref{cor:compare ps and PC} and \cref{lem: compare perm and sym}). Thus a consequence of \cref{intro thomason result} is that the covariant version Bohmann--Osorno's $K$-theory machine is essentially surjective, proving a conjecture of Elmendorf \cite{Elmendorf2021}. 

Previous approaches to this conjecture have considered an inverse $K$-theory construction similar to Mandell's. The immediate obstruction to confirming this conjecture arises from needing to check some technical properties of inverse $K$-theory, although some of the necessary properties of $K^{-1}$ have been verified by Johnson--Yau \cite{JohnsonYau/invKT} (see also \cite{Elmendorf2021} and \cite[Section 10.3]{yau:21}). Rather than considering multifunctorial properties of inverse $K$-theory, we first prove an $\infty$-categorical version of \cref{intro thomason result} (see \cref{lem:Mack ps Perm is infty Mack}), and produce a comparison with Bohmann--Osorno's $K$-theory. 

\begin{introtheorem}[\cref{thm:main}]\label{introthm:KTs agree}
The functor $\mathbb{K}$ and its $\infty$-categorical analogue produce equivalent genuine $G$-spectra.
\end{introtheorem}


Each of the categories of Mackey functors comes with a notion of \textit{stable equivalence}, which are morphisms of Mackey functors that become levelwise equivalences after $K$-theory. We show that upon formally inverting these stable equivalences, there is no difference between some notions of categorical Mackey functor. Moreover, these localizations model the subcategory of connective spectra within the equivariant stable homotopy category.

\begin{introtheorem}[\cref{lem:Mack ps Perm is infty Mack,cor:Mack Mandell,cor:str becomes equiv after loc}]\label{introthm:eq Thomason}
Equivariant $K$-theory induces equivalences of homotopy categories between
\begin{itemize}
    \item[--] the homotopy theory of symmetric monoidal Mackey functors;  
    \item[--] the homotopy theory of symmetric monoidal $\infty$-Mackey functors;
    \item[--] the homotopy category of connective genuine $G$-spectra.
\end{itemize}
\end{introtheorem}

In particular, as stated in \cref{intro thomason result}, our work yields an equivariant version of Thomason's theorem. Other equivariant versions of Thomason's theorem have been proved by Lenz \cite{lenz:Gglobal,lenz:symmonGcat} for different definitions of equivariant symmetric monoidal category than the ones we consider in this paper.  Namely, Lenz models equivariant symmetric monoidal categories by either $G$-objects in symmetric monoidal categories or operads over the genuine equivariant Barratt--Eccles operad of \cite{GuillouMay:PermutativeGCats}. It would be interesting to directly compare his results with ours, although we do not pursue this comparison here.

In addition to providing an equivariant analogue of Thomason's theorem, another goal of this paper is to lay out an approach to symmetric monoidal structures in the context of equivariant algebra. These structures have seen a tremendous amount of interest recently, due in large part to the role they play in Hill, Hopkins, and Ravenel's solution to the Kervaire invariant one problem \cite{HHR}.  In this paper we address, in a systematic way, the kinds of equivariant categorical structures which arise in the context of \cite{HHR} and equivariant homotopy theory more broadly. In future work, we hope to investigate how these structures interact with equivariant ring spectra. Although \cref{introthm:eq Thomason} shows that equivariant $K$-theory induces an equivalence between symmetric monoidal Mackey functors and connective genuine $G$-spectra, we know it does not preserve multiplicative structure: $K$-theory sends ring objects in symmetric monoidal Mackey functors to spectral Green functors, which do not represent genuine equivariant ring spectra (see \cref{remark: green not tambara}).

\subsection{Outline}
In \cref{sec:infty inv KT}, we briefly review the story of inverse $K$-theory for permutative categories and recontextualize it in the language of $\infty$-categories (see, in particular, \cref{thm:Mandell for infty cats} and \cref{cor:equiv smc and sp}).
\cref{sec:Mack} describes the different approaches to Mackey objects: Bohmann--Osorno's categorical Mackey functors (\cref{sec:cat Mack}), permutative Mackey functors (\cref{sec:perm Mack}), $\infty$-categorical Mackey functors (\cref{sec:infty Mack}), and spectral Mackey functors (\cref{sec:Sp Mack}).
The first discussion of equivariant $K$-theory is in \cref{sec:main}, where we revisit Bohmann--Osorno's $K$-theory machine. The reader who is primarily interested in \cref{intro thomason result} may prefer to read the introduction to \cref{sec:Mack} and then proceed directly to \cref{sec:main}, where we state and prove the main results. There is one appendix, \cref{app:GM thm}, which gives the proof of the covariant Guillou--May theorem (\cref{thm:GM no op}). 

\subsection{Notation and conventions}

We fix a finite group $G$ throughout.
By an $\infty$-category, we mean a quasicategory in the sense of \cite{HTT, HA}.
Since this paper compares several different frameworks for the same ideas, we will need to introduce some non-standard notation.
\begin{itemize}
    \item $\mathcal{A}^G$ denotes the $1$-category whose objects are finite $G$-sets and morphisms are isomorphism classes of spans, and $\Mack(-) = \Fun^{\rm add}(\mathcal{A}^G, -)$. 
    \item $\mathcal{A}^G_2$ denotes the $(2,1)$-category whose objects are finite $G$-sets, 1-morphisms are spans of finite $G$-sets, and $2$-morphisms are isomorphisms of spans; and $\Mack_{\ps}(-) = \Fun^{\rm add}_{\ps}(\mathcal{A}^G_2,-)$ where $\Fun^{\rm add}_{\ps}$ denotes strictly unital (lax) additive pseudo-functors.  In this paper, we will only consider pseudo-functors which are strictly unital.
    \item $\mathcal{A}^G_{\PC}$ denotes the $\PC$-enriched Burnside category, and $\Mack_{\PC}(-) =\Fun^{\rm add}_{\PC}(\mathcal{A}^G_{\PC}, -)$.
    \item $\mathcal{A}^G_{\rm Sp}$ denotes the spectrally enriched Burnside category, and $\Mack_{\Sp}(-)=\Fun^{\rm add}_{\Sp}(\mathcal{A}^G_{\Sp},-)$.
    \item For any $2$-category $\cat C$, let $N_*^D(\cat C)$ denote the Duskin nerve of $\cat C$. The \emph{effective Burnside category} is $\mathcal{A}^G_{\rm eff} = N_*^D(\mathcal{A}^G_2)$, and $\Mack_\infty(-) = \Fun_\infty(\mathcal{A}^G_{\rm eff}, -)$.
    \item $\Perm_2$ denotes the (2,1)-category of permutative categories, strictly unital strong monoidal functors, and natural isomorphisms; $\Perm_{\PC}$ denotes the $\PC$-enriched category of permutative categories; and $\Perm_{\Sp}$ denotes the spectrally enriched category of permutative categories.
    \item  We use $\Sp_\infty$ to denote the $\infty$-category of spectra, and $\Sp$ to denote the category of orthogonal spectra \cite{MM}; $h\Sp$ denotes the stable homotopy category and $h\Sp_{\geq 0}$ denotes the full subcategory spanned by connective spectra. 
    \item We write $\Sp^G$ to denote the category of orthogonal $G$-spectra indexed on a complete universe, $\Sp^G_{\infty}$ for the associated $\infty$-category, and $h\Sp^{G}$ for the equivariant stable homotopy category.
\end{itemize}

\subsection*{Acknowledgements} We would like to thank Liam Keenan for their insights that contributed to the development of this paper. Additionally, we are very grateful to Niles Johnson for pointing out the subtleties surrounding duality and enriched Burnside categories.
We would also like to thank Anna Marie Bohmann, Daniel Fuentes-Keuthan, Elijah Gunther, Andres Mejia, Mona Merling, Lyne Moser, and Ang\'elica Osorno for helpful conversations.
Additionally, we are grateful to Michigan State University for hosting the first-named author and facilitating this collaboration. The first-named author was also partially supported by NSF grant DGE-1845298.  The second-named author was partially supported by NSF grant DMS-2135960.
Some ideas of this work, specifically those related to the definition of symmetric monoidal Mackey functors, also appear in the second-named author's thesis \cite{davidthesis}, which was partially supported by NSF-DMS grant 2104300.

\section{Reframing inverse \texorpdfstring{$K$}{TEXT}-theory in \texorpdfstring{$\infty$}{TEXT}-categories}\label{sec:infty inv KT}

The $K$-theory machines of Segal \cite{segal74} and May \cite{may78} produce connective spectra from symmetric monoidal and permutative categories, respectively. Thomason showed in \cite{Thomason95} that all connective spectra arise as the $K$-theory of symmetric monoidal categories, up to equivalence. Mandell \cite{Mandell/invKT} gave another proof of Thomason's result using permutative categories, which are a stricter version of symmetric monoidal categories.

In this section, we review these ideas and recontextualize them in the language of $\infty$-categories. Since Bohmann--Osorno build their equivariant $K$-theory machine using permutative categories, it will be convenient for us to consider both the permutative and symmetric monoidal story here. 

\begin{definition}
    A \textit{permutative category} $(\cat C,\oplus,0)$ is small category with a strictly associative, symmetric binary functor $\oplus\colon \cat{C}^2\to \cat{C}$, and a choice of object $0$ which is a strict unit for $\oplus$.  A \textit{strictly unital lax monoidal functor} of permutative categories is a functor $f\colon \cat C\to \cat D$ such that $f(0_\mathcal{C}) =0_{\mathscr{D}}$, together with a natural transformation 
    \[
        \delta \colon f(c_1)\oplus f(c_2) \to f(c_1\oplus c_2)
    \]
    such that the usual associativity and  unitality diagrams commute.

    A lax monoidal functor is called \textit{strong monoidal} if it is strictly unital and if $\delta$ is a natural isomorphism.
    A \textit{natural transformation of lax functors} is a natural transformation which commutes with the structure maps.
\end{definition}

We write $\Perm$ for the $1$-category of permutative categories and strong monoidal functors.  We let $\Perm_2$ denote the (strict) $(2,1)$-category of permutative categories, strong monoidal functors, and natural isomorphisms. Similarly, we write $\Sym$ for the $1$-category of symmetric monoidal categories and strong monoidal functors, and $\Sym_2$ for the corresponding $(2,1)$-category.

The $K$-theory functors from permutative and symmetric monoidal categories can be factored as\[
K\colon \Sym \xrightarrow{\ell} \Perm \xrightarrow{\hat K} \Gamma\text{-Cat} \xrightarrow{N_*}\Gamma\text{-sSet},
\] where $\ell$ is a strictification functor (c.f. \cite{Isbell}) and $N_*$ is the nerve functor $\mathrm{Cat}\to \mathrm{sSet}$ applied levelwise. Here, $\Gamma$-sSet is Segal's model of the category of connective spectra \cite{segal74}. In particular, two special $\Gamma$-simplicial sets are said to be \textit{stably equivalent} if they produce equivalent spectra. Using $K$, $\Sym$ and $\Perm$ can also be equipped with a notion of stable equivalence.

\begin{definition}\label{defn:st equiv in perm}
    Two permutative or symmetric monoidal categories are \textit{stably equivalent} if they are stably equivalent in $\Gamma$-sSet after applying $K$. Let $W_K$ denote the collection of stable equivalences in $\Perm$ or $\Sym$.
\end{definition}

\begin{theorem}[Thomason, Mandell]\label{thm:Mandell for infty cats}
    The $K$-theory functor induces equivalences of $\infty$-categories\[\begin{tikzcd}
             N_*(\Sym)[W_K^{-1}] \ar[r, shift left, swap] & N_*(\Perm)[W_K^{-1}] \ar[r, shift left] \ar[l, shift left,] & \Sp_\infty^{\geq 0} \ar[l, shift left] 
    \end{tikzcd}
    \] where $\Sp_\infty^{\geq 0}$ is the $\infty$-category of connective spectra. In particular, this induces an equivalence of homotopy categories.
\end{theorem}

\begin{proof}
By \cite[Theorem 0.10]{MMSS}, the Dwyer--Kan localization $N_*(\Gamma\text{-sSet})[W^{-1}]$ with respect to stable equivalences $W$ is equivalent to the $\infty$-category of connective spectra. This allows us to replace the category of connective spectra in the statement with $N_*(\Gamma\text{-sSet})[W^{-1}]$. The result then follows from the work of Mandell \cite{Mandell/invKT}, who constructs an adjunction between $\Perm$ and $\Gamma\text{-sSet}$ where both functors preserve stable equivalences, and the unit and counit of the adjunction are valued in stable equivalences. By \cite[Theorem 3.9]{Mandell/invKT} (see also \cite[Lemma 1.9.2]{Thomason95}), this also gives the result for symmetric monoidal categories.
\end{proof}

There are several options for morphisms in the categories of symmetric monoidal or permutative categories.  For example, one could consider lax, strong, or even strict monoidal functors. However, lax symmetric monoidal functors of permutative categories can be replaced with strong monoidal functors of symmetric monoidal categories, up to stable equivalence, see \cite[Lemma 1.9.2]{Thomason95} or \cite[Theorem 3.9]{Mandell/invKT}. In any case, the above theorem will remain true with the choice of either lax or strong monoidal functors.

We can reprove the theorem completely internally in the language of $\infty$-categories using symmetric monoidal $\infty$-categories. 
In this context, $K^{-1}$ has a simple description.

Let $\smc$ be the large $\infty$-category of symmetric monoidal $\infty$-categories (with strong monoidal functors), as in \cite[2.1.4.13]{HA}.
As shown in \cite[\S 8]{GMN}, there is an extension of  Segal's definition of algebraic $K$-theory to symmetric monoidal $\infty$-categories:
\[
\begin{tikzcd}
    K\colon\smc \ar{r}{(-)^\simeq} &  \operatorname{Alg}_{\mathbb{E}_\infty}({\cat  S}) \ar{r}{\Omega B} & \operatorname{Grp}_{\mathbb{E}_\infty}({\cat S}) \simeq \Sp^{\geq 0}_\infty,
\end{tikzcd}
\] where $(-)^{\simeq}$ denotes taking the maximal $\infty$-groupoid and $\Omega B$ is group completion.
We can thus define the class of \textit{stable equivalences} $W_\infty$ in $\smc$ to be strong symmetric monoidal functors $\mathscr{C}\rightarrow \mathscr{D}$ that induce equivalences $K(\mathscr{C})\rightarrow K(\mathscr{D})$ in spectra.
The definition of $K$-theory fits into the zig-zag of adjunctions in $\infty$-categories:
\begin{equation}\label{equation: K-theory zig-zag}
\begin{tikzcd}
	\smc \\
	\\
	{\operatorname{Alg}_{\mathbb{E}_\infty}(\mathscr{S})} && {\operatorname{Grp}_{\mathbb{E}_\infty}(\mathscr{S})} & {\Sp^{\geq 0}_\infty.} \ar{l}{\Omega^\infty}[swap]{\simeq}
	\arrow["{(-)^\simeq}", shift left, curve={height=-6pt}, from=1-1, to=3-1]
	\arrow[shift left, curve={height=-6pt}, hook, from=3-1, to=1-1]
	\arrow[draw=none, from=3-1, to=1-1]
	\arrow["{\Omega B}", curve={height=-12pt}, from=3-1, to=3-3]
	\arrow[curve={height=-12pt}, hook', from=3-3, to=3-1]
	\arrow["\perp"{description}, draw=none, from=3-1, to=3-3]
	\arrow[draw=none, from=3-1, to=1-1]
	\arrow["\dashv"{description}, draw=none, from=3-1, to=1-1]
\end{tikzcd}
\end{equation}
The vertical adjoint pair stems from the adjunction
\[
\begin{tikzcd}
   \hspace{0.6cm}{\mathscr{S}} \ar[bend left, hook]{rr}\arrow[rr, phantom, "\perp" description]  & & {\mathscr{C}\textup{at}_\infty} \ar[bend left]{ll}{(-)^\simeq}
\end{tikzcd}
\]
for which the left adjoint is given by the fundamental $\infty$-groupoid, or alternatively as the inclusion of $\infty$-groupoids into $\infty$-categories. This inclusion is strong monoidal with respect to the Cartesian monoidal structure, and therefore the adjunction lifts to $\mathbb{E}_\infty$-algebras and moreover $\operatorname{Alg}_{\mathbb{E}_\infty}(\mathscr{C}\textup{at}_\infty)\simeq \smc$.
Inverse $K$-theory is the composite of the forgetful functors above:
\[
\begin{tikzcd}
K^{-1}\colon \Sp^{\geq 0}_\infty \simeq \operatorname{Grp}_{\mathbb{E}_\infty}(\mathscr{S}) \ar[hook]{r} & \operatorname{Alg}_{\mathbb{E}_\infty}(\mathscr{S}) \ar[hook]{r} & \smc.
\end{tikzcd}
\]

\begin{corollary}\label{cor:equiv smc and sp}
The functor $K\colon \smc\rightarrow \Sp^{\geq 0}_\infty$ induces an equivalence of $\infty$-categories
\[
\smc[W^{-1}_\infty]\simeq \Sp^{\geq 0}_\infty.
\]
In particular,
\[\smc [W^{-1}_\infty] \simeq  N_*(\Perm)[W_K^{-1}]\simeq N_*(\Sym)[W_K^{-1}],
\]
so for any symmetric monoidal $\infty$-category $\mathscr{C}$, there is a permutative category $C$ with $K(\mathscr{C})\simeq K(C)$.
\end{corollary}
\begin{proof}
The proof is an extension of Thomason's arguments \cite{Thomason95} with the added flexibility of the language of $\infty$-categories. Essentially, localizing guarantees that the units and counits of the adjunctions in \cref{equation: K-theory zig-zag} are equivalences. 

In more detail, note that every symmetric monoidal $\infty$-category $\cat C$ is stably equivalent to its maximal $\infty$-groupoid, essentially by definition. It thus suffices to consider $\infty$-groupoids. 
Suppose now that $\mathscr{C}$ is a symmetric monoidal $\infty$-groupoid, so that $\mathscr{C}^\simeq\simeq \mathscr{C}$. Then $K(\mathscr{C})\simeq \Omega B \mathscr{C}$, and since $\Omega B(\Omega B \mathscr{C})\simeq \Omega B \mathscr{C}$, we see that $\mathscr{C}\rightarrow K^{-1}(K(\mathscr{C}))$ is a stable equivalence for any symmetric monoidal $\infty$-groupoid $\mathscr{C}$, and hence for any symmetric monoidal $\infty$-category by our previous discussion. Analogously, we obtain an equivalence $K(K^{-1}(E))\rightarrow E$ for any connective spectrum $E$. This proves the first statement in the theorem, and the second statement follows from \cref{thm:Mandell for infty cats}.
\end{proof}

\section{Variations on a theme: Mackey functors}\label{sec:Mack}

In this section we define symmetric monoidal Mackey functors and the relevant variations. 
Mackey functors were first defined by Dress to axiomatize structures in representation theory \cite{Dress:Mackey}, and play the role of Abelian groups in equivariant homotopy theory. 
They also provide a useful framework for equivariant topology because the genuine equivariant cohomology of a $G$-space is naturally a Mackey functor.  

Each of the variations of Mackey functors that we consider requires a corresponding variation of the \textit{Burnside category}, described below. We detail the variations on this category that we need in the following subsections.

\begin{definition}
    The \textit{Burnside category} $\mathcal{A}^G$ is the $1$-category whose objects are finite $G$-sets and whose morphisms are isomorphism classes of \textit{spans} of $G$-sets $S\leftarrow U\rightarrow T$. Two spans are isomorphic if there is a commutative diagram of $G$-sets\[
    \begin{tikzcd}[column sep=1cm,row sep=0.1cm]
     & U \ar[ld] \ar[rd]\ar[dd, "\cong"] &\\
        S & & T\\
        & U' \ar[ul]\ar[ur] &
    \end{tikzcd}.
    \] Composition of spans is given by pullback, and $\mathcal{A}^G$ comes equipped with a preadditive structure via disjoint union.
\end{definition}

\begin{definition}
    A \textit{Mackey functor} is a product-preserving functor $M\colon \mathcal{A}^G\to {\rm Ab}$, where ${\rm Ab}$ is the category of Abelian groups.
\end{definition}

Essentially, a Mackey functor $M$ for a finite group $G$ consists of a collection of Abelian groups $\{M(H)\}$ indexed on the subgroups $H\leq G$, connected by a system of additive operations called restrictions and transfers. More details on equivariant algebra can be found in \cite{ThevenazWebb:1995} or \cite{lewis:1980}.

More generally, Mackey functors can be valued in preadditive categories other than Abelian groups. We let $\Mack(\cat C) = \Fun^{\rm add}(\mathcal{A}^G, \cat C)$ denote the $1$-category of Mackey functors valued in $\cat C$. The examples we will discuss in this paper are the following:\begin{itemize}
    \item \cref{sec:cat Mack}: \textit{$\PC$-Mackey functors}, called categorical Mackey functors by Bohmann--Osorno \cite{BohmannOsorno:2015}, are valued in permutative categories and can be thought of as a collection of permutative categories $\{\cat C(H)\}_{H\leq G}$ whose restrictions and transfers are strictly unital lax functors. 
    \item \cref{sec:perm Mack}: \textit{Permutative Mackey functors} are like $\PC$-Mackey functors, except the restrictions and transfers are strong monoidal.  Additionally, there are more morphisms between permutative Mackey functors and the compatibility conditions on restrictions and transfers are only required to hold up to natural isomorphism (as opposed to strict equality).
    \item \cref{sec:perm Mack}: \textit{Symmetric monoidal Mackey functors} are the same as permutative Mackey functors, but with symmetric monoidal categories instead of permutative categories.
    \item \cref{sec:infty Mack}: By work of Barwick and coauthors Glasman--Shah \cite{BarwickOne, Barwick2}, there is an $\infty$-categorical version of the Burnside category and consequently \textit{$\infty$-categorical Mackey functors} valued in a preadditive $\infty$-category.
    \item \cref{sec:Sp Mack}: \textit{Spectral Mackey functors} are collections of spectra $\{\cat E(H)\}_{H\leq G}$ whose restriction and transfer maps are maps of spectra. Guillou--May show in \cite{GuillouMay:2011} that spectral Mackey functors provide a model for genuine $G$-spectra; a genuine $G$-spectrum $E$ determines a spectral Mackey functor which at level $H\leq G$ is the categorical $H$-fixed points of $E$.
\end{itemize}

\begin{remark}
    Mackey functors are sometimes defined as functors out of the opposite category of $\mathcal{A}^G$. If we do not consider the opposite, the resulting definition is the same since the Burnside category is self-dual, as spans can always be ``turned around.'' This self-duality holds for many variations of the Burnside category (e.g.\ the effective Burnside category of \cite{BarwickOne}), however, in the enriched setting, duality (taking opposite categories) is not always compatible with enrichment. In \cref{app:GM thm}, we show that (for the purposes of spectral Mackey functors) there is no meaningful difference between the contravariant definition in \cite{GuillouMay:2011} and the covariant definition we use here.
    
\end{remark}

\subsection{\texorpdfstring{$\PC$}{TEXT}-Mackey functors}\label{sec:cat Mack}

In this subsection, we recall Bohmann--Osorno's definition of categorical Mackey functors, which we call $\PC$-Mackey functors. To describe $\PC$-Mackey functors, we need some preliminary definitions; more detailed descriptions may be found in \cite{BohmannOsorno:2015}. A $\PC$-category is a category enriched in permutative categories with some extra structure coming from \textit{multilinear} functors. For our purposes, the following definition of \textit{bilinear} functors suffices.

\begin{definition}
    Let $\cat{C}_i$ be permutative categories for $i=1, 2, 3$.  A \emph{bilinear functor}
$
    f\colon \cat{C}_1\times \cat{C}_2\to \cat{C}_3 
$
is a functor from the product category $\cat{C}_1\times \cat{C}_2$ to $\cat{C}_3$ together with natural transformations
\begin{align*}
    \delta_{2}\colon f(x,y)\oplus f(x,z) & \to f(x,y\oplus z)\\
    \delta_{1}\colon f(y,x)\oplus f(z,x) & \to f(y\oplus z, x)
\end{align*}
satisfying the conditions of \cite[Definition 3.4]{Elmendorf-Mandell:06}.  These maps should be thought of as bilinear maps in the sense of ordinary linear algebra.
\end{definition}

The category of permutative categories can be given the structure of a \emph{multicategory}, a version of a category equipped with a notion of multilinear maps; multicategories are also sometimes called \textit{colored operads}. Let $\PC$ denote the multicategory of permutative categories (see \cite[\S 3]{BohmannOsorno:2015} for a full definition). This extra structure is important for the $K$-theory construction of Bohmann--Osorno since it allows for effective change of enrichment from functors enriched in permutative categories to functors enriched in spectra. 

\begin{definition}[{cf.\ \cite[Definition 4.1]{BohmannOsorno:2015} and the discussion that follows.}]\label{PC categories definition}
A $\PC$-category $\cat{C}$ consists of the following data:
\begin{enumerate}
    \item[(i)] a collection of objects $\mathrm{ob}(\cat{C})$,
    \item[(ii)] a permutative category $\cat{C}(x,y)$ for all $x,y \in \mathrm{ob}(\cat{C})$,
    \item[(iii)] an identity object $\id_x\in \cat{C}(x,x)$,
    \item[(iv)] a bilinear functor
    \[
        \circ_{x,y,z}\colon \cat{C}(y,z)\times \cat{C}(x,y)\to \cat{C}(x,z)
    \]
    such that the identity objects are strict identities.
\end{enumerate}
When the objects are clear, we will write $\circ$ instead of $\circ_{x,y,z}$.
\end{definition}

\begin{example}[{c.f.\ \cite[Section 5]{BohmannOsorno:2015}}]
    The $2$-category of small permutative categories, strictly unital lax monoidal functors, and monoidal natural transformations is a $\PC$-category, denoted $\Perm_{\PC}$.  The permutative structure on strictly unital lax monoidal functors is given by pointwise addition. 
\end{example}

\begin{definition}\label{PC functor definition}
    Let $\cat{C}$ and $\cat{D}$ be $\PC$-categories.  A $\PC$-functor $F\colon \cat{C}\to \cat{D}$ consists of an object function $\mathrm{ob}(\cat{C})\to \mathrm{ob}(\cat{D})$ as well as strictly unital lax functors 
$
        F_{x,y}\colon \cat{C}(x,y)\to \cat{D}(Fx,Fy)
$
    such that the following diagram commutes for all $x,y,z\in \cat{C}$
    \[
        \begin{tikzcd}
            \cat{C}(y,z)\times\cat{C}(x,y) \ar[r,"\circ"] \ar[d, swap, "F_{y,z}\times F_{x,y}"] & \cat{C}(x,z) \ar[d,"F_{x,y}"] \\
            \cat{D}(Fy,Fz)\times \cat{D}(Fx,Fy) \ar[r,swap, "\circ"] & \cat{D}(Fx,Fz).
        \end{tikzcd}
    \]
\end{definition}

\begin{definition}\label{PC NT definition}
Let $\cat{C}$ and $\cat{D}$ be $\PC$-categories and let $F_1, F_2\colon \cat{C}\to \cat{D}$ be two $\PC$-functors. A \emph{$\PC$-natural transformation} $\alpha\colon F_1\Rightarrow F_2$ consists of a collection of morphisms 
    \[
        \alpha_{x}\colon F_1x\to F_2x
    \]
in $\cat{D}$ so that for any $f\colon x\to z$ in $\cat{C}$ the following diagram commutes
\[
    \begin{tikzcd}
    \cat{C}(x,y) \ar[d,"F_1"] \ar[r,"F_2"] & \cat{D}(F_2x,F_2y)  \ar[d,"\alpha_x^*"]  \\
    \cat{D}(F_1x,F_1y) \ar[r,"(\alpha_{y})_*"]& \cat{D}(F_1x,F_2y).
    \end{tikzcd}
\]
\end{definition}

Given two $\PC$-categories $\cat{C}$ and $\cat{D}$, there is a category of $\PC$-functors and natural transformations we denote by $\Fun_{\PC}(\cat{C},\cat{D})$.  In \cite[Theorem 7.5]{BohmannOsorno:2015}, the authors use a category of $\PC$-functors as input for a $K$-theory machine. The domain of these enriched functors is a $\PC$-enriched version of the Burnside category.

\begin{definition}
    The $\PC$-Burnside category $\mathcal{A}^G_{\PC}$ has objects finite $G$-sets, and the permutative category of morphisms $\mathcal{A}_{\PC}^G(S,T)$ consists of spans $S\leftarrow U \to T$ of finite $G$-sets and isomorphisms of spans. The permutative structure is given by disjoint union.
\end{definition}

Composition in $\mathcal{A}^G_{\PC}$ is given by pullback, which can be made strictly unital and associative by choosing specific models, as discussed in the next remark (see also \cite[Definition 7.2]{BohmannOsorno:2015}).

\begin{remark}
We will implicitly take our $G$-sets to be totally ordered sets on which $G$ acts through permutations. An equivariant map between $G$-sets may or may not respect this total ordering.  Forgetting the total ordering, the category of finite $G$-sets is equivalent to the usual category of finite $G$-sets. 

By introducing a total ordering on $G$-sets, we can give a strictly associative and unital model of disjoint union for finite $G$-sets. In particular, the category of finite $G$-sets and disjoint unions, and related categories, are permutative instead of merely symmetric monoidal.
\end{remark}

\begin{definition}
    A \textit{$\PC$-Mackey functor} is a $\PC$-enriched functor $\mathcal{A}_{\PC}^G\to {\Perm}_{\PC}$. We denote the category of $\PC$-Mackey functors by \[\Mack_{\PC}(\Perm_{\PC}):=\Fun_{\PC}(\mathcal{A}_{\PC}^G, {\Perm}_{\PC}).\]
\end{definition}

Bohmann--Osorno call these objects categorical Mackey functors. We have opted for a different name to illustrate the difference between this definition and the one in the next subsection.

\subsection{Symmetric monoidal Mackey functors}\label{sec:perm Mack}
In this subsection, we define another notion of categorical Mackey functor using pseudo-functors, following ideas of \cite{HillHopkins}.  The advantage of this new construction is that pseudo-functors tend to be more commonly occurring ``in nature'' than strict functors and thus it is easier to find examples; indeed, every example we know of can be constructed this way. Moreover, these objects are easier to compare with the $\infty$-categorical version of Mackey functors discussed in \cref{sec:infty Mack}. 

\begin{definition}
Suppose that $\cat C$ and $\cat D$ are $(2,1)$-categories.  A \textit{pseudo-functor} $F\colon \cat{C}\to \cat{D}$ consists of an object function $\mathrm{ob}(\cat{C})\to \mathrm{ob}(\cat{D})$, together with functors $\cat{C}(x,y)\to \cat{D}(Fx,Fy)$ for all objects $x,y\in \cat{C}$. These functors must be compatible with units and composition up to some invertible $2$-cells in $\mathscr{D}$ satisfying some coherence relations; see \cite[Definition 4.1.2]{JohnsonYau2Cat}.  A \emph{strict functor} is a pseudo-funtor for which all the defining $2$-cells are identities.
\end{definition}

In this paper we will only use pseudo-functors which are \emph{strictly unital}, meaning they preserve identity morphisms on the nose.

\begin{remark} \label{rem:lax versus pseudo}
In general, one often considers \emph{lax} functors which are like pseudo-functors except the structure $2$-cells for units and composition are not assumed to be invertible. All of the $2$-categories we consider in this paper will be $(2,1)$-categories, where every lax functor is automatically a pseudo-functor.
\end{remark}

Between two pseudo-functors there is a notion of pseudo-natural transformation, which is just like an ordinary natural transformation except the defining diagrams again only hold up to some specified invertible $2$-cells.  The collection of pseudo-functors from $\cat C$ to $\cat D$ and pseudo-natural transformations between them forms a $1$-category we denote by $\Fun_{\ps}(\cat C,\cat D)$. 
Similarly, we can consider the $1$-category of strict $2$-functors and strict $2$-natural transformations $\Fun_2(\mathscr{C},\mathscr{D})$.

\begin{definition}
    The \textit{$(2,1)$-Burnside category} $\mathcal{A}^G_2$ has objects finite $G$-sets. Given finite $G$-sets $S$ and $T$, ${\rm Hom}(S,T)$ is the groupoid of spans\[
     S\longleftarrow U \longrightarrow T
\] of finite $G$-sets and isomorphisms of spans. Horizontal composition in this $2$-category is given by pullback of spans. Disjoint union of finite $G$-sets induces a preadditive structure on $\mathcal{A}^G_2$.
\end{definition}

\begin{definition}
    We write $\Perm_2$ for the $(2,1)$-category of permutative categories, strictly unital strong monoidal functors, and monoidal natural isomorphisms.
\end{definition}

Both of the $2$-categories $\Perm_2$ and $\mathcal{A}^G_2$ admit $2$-categorical products given by the cartesian product of permutative categories and disjoint union, respectively. We will call a pseudo-functor between these two $2$-categories \emph{additive} if it preserves these products, up to equivalence.

\begin{definition}
    A \textit{permutative Mackey functor} is a strictly unital additive pseudo-functor $\mathcal{A}^G_2\to\Perm_2$. A morphism of permutative Mackey functors is a pseudo-natural transformation. We let $
     \Mack_{\ps}(\Perm_2):=\Fun^{\rm add}_{\ps}(\mathcal{A}^G_2,\Perm_2)$ denote the category of permutative Mackey functors.
\end{definition}

Our goal is to show that every permutative Mackey functor determines a $\PC$-Mackey functor.  As a first step, we show how to produce a $\PC$-Mackey functor out of a strict $2$-functor. Recall that $\Fun_2(\mathcal{A}^G_2, \Perm_2)$ denotes the category of strict $2$-functors and strict $2$-natural transformations. 

\begin{lemma}\label{lem: Eckmann--Hilton}
        If $F\colon \mathcal{A}^G_2\to \Perm_2$ is strict and additive, then the functors
        \[
            F\colon \mathcal{A}^G_2(X,Y)\to \Perm_2(F(X),F(Y))
        \]
        are strong monoidal.
    \end{lemma}
    \begin{proof}
        Let $\omega_i\colon X\to Y$ be $1$-cells in the Burnside category, for $i=1,2$.  We need to show there are natural isomorphisms
        \[
            F(\omega_1)\oplus F(\omega_2)\simeq F(\omega_1+\omega_2)
        \]
        where $F(\omega_1)\oplus F(\omega_2)\colon F(X)\to F(Y)$ is defined as
        \[
            (F(\omega_1)\oplus F(\omega_2))(x) = F(\omega_1)(x)+F(\omega_2)(x)
        \]
        where $+$ is the monoidal product in $F(Y)$.  The map $\omega_1+\omega_2$ is the composite 
        \[
            X \xrightarrow{R_{\nabla_X}} X\amalg X \xrightarrow{\omega_1\amalg \omega_2} Y\amalg Y \xrightarrow{T_{\nabla_Y}} Y
        \]
        where $\nabla$ is a fold map for either $X$ or $Y$, $R$ is a restriction, and $T$ is a transfer.  Since $F$ preserves products, $F(R(\nabla_X))\colon F(X) \to F(X\amalg X)$ is naturally isomorphic to the diagonal functor, as composition with either projection gives the identity. Additionally, $F(T_{\nabla_Y})\colon F(Y\amalg Y)\to F(Y)$ is naturally isomorphic to $+$; to see this, observe the map $F(T_{\nabla_Y})$ endows $F(Y)$ with a monoidal product which, a priori, is different from $+$. However, as $F(T_{\nabla})$ is strong monoidal (with respect to $+$) this new monoidal product distributes over the original one and so by an Eckmann--Hilton style argument these two products are the same up to natural isomorphism. Consequently, the diagram
        \[
        \begin{tikzcd}[column sep = huge]
        F(X) \ar[r,"F(R_{\nabla_X})"] \ar["\Delta",dr,swap]& F(X\amalg X)\ar[d,"\simeq"]\ar[r,"F(\omega_1\amalg \omega_2)"] & F(Y\amalg Y) \ar[d,"\simeq"] \ar[r,"F(T_{\nabla_Y})"] & F(Y) \\
        & F(X)\times F(X) \ar[r,"F(\omega_1)\times F(\omega_2)"]& F(Y)\times F(Y) \ar[ur,swap,"+"]
        \end{tikzcd}
        \]
        commutes up to natural isomorphism, as the middle square commutes by additivity of $F$. Since the top row of the diagram is $F(\omega_1+\omega_2)$ and the bottom composite is $F(\omega_1)\oplus F(\omega_2)$, the proof is complete.
    \end{proof}

\begin{proposition}\label{prop: strict to PC}
    There is a fully faithful functor
    \[
        \Psi\colon \Fun^{\rm add}_{2}(\mathcal{A}^G_2,\Perm_2)\hookrightarrow\Mack_{\PC}(\Perm_{\PC}).
    \]
\end{proposition}
\begin{proof}
    This is simply an unwinding of the definitions.  A $\PC$-functor $F\colon \mathcal{A}^G_{\PC}\to \Perm_{\PC}$ consists of a permutative category $F(X)$ for all finite $G$-sets $X$, together with strictly unital lax monoidal functors
    \begin{equation}\label{eq:PC lax bit}
        F_{X,Y}\colon \mathcal{A}_{\PC}^G(X,Y)\to \Perm_{\PC}(F(X),F(Y)) 
    \end{equation}
    such that $F_{X,X}(\mathrm{id_X}) = \mathrm{id}_{F(X)}$ and the diagram
    \begin{equation}\label{PC enriched functor diagram}
        \begin{tikzcd}
            \mathcal{A}_{\PC}^G(Y,Z)\times\mathcal{A}_{\PC}^G(X,Y) \ar[r,"\circ"] \ar[d, swap, "F_{Y,Z}\times F_{X,Y}"] & \mathcal{A}_{\PC}^G(X,Z) \ar[d,"F_{X,Z}"] \\
            \Perm_{\PC}(F(Y),F(Z))\times\Perm_{\PC}(F(X),F(Y)) \ar[r,swap, "\circ"] & \Perm_{\PC}(F(X),F(Z))
        \end{tikzcd}
    \end{equation}
    commutes.  An additive strict $2$-functor $F\colon \mathcal{A}^G_2\to \Perm_{2}$ is exactly the same data, except that, by \cref{lem: Eckmann--Hilton}, the functors \eqref{eq:PC lax bit} are strong monoidal instead of simply lax monoidal.  The morphisms are identical.
\end{proof}

\begin{remark}\label{rmk:psi not ess surj}
    The functor $\Psi$ in the last proposition is not essentially surjective since there can be $\PC$-Mackey functors $F$ for which the functors \eqref{eq:PC lax bit} are not strong monoidal.  That being said, all concrete examples of $\PC$-Mackey functors considered in \cite{BohmannOsorno:2015} are in the essential image of $\Psi$. 
\end{remark}

We now compare the strict and pseudo-functor categories by a standard strictification argument, adjusted to account for the permutative and additive requirements. There is an evident inclusion\[
        i\colon \Fun^{\rm add}_{2}(\mathcal{A}^G_{2},\Perm_{2})\hookrightarrow \Mack_{\ps}(\Perm_{2}),
    \] 
and we define a \emph{strictification functor} \[j\colon \Mack_{\ps}(\Perm_2)\to \Fun_{2}^{\mathrm{add}}(\mathcal{A}^G_2,\Perm_2)\]
which runs counter to $i$.

\begin{definition}\label{defn: strictification functor}
    Let $F\colon \mathcal{A}^G_{2}\to \Perm_2$ be an additive pseudo-functor, and define an additive $2$-functor $j(F)\colon \mathcal{A}^G_2\to \Perm_2$ as follows: for any finite $G$-set $X$, let $j(F)(X)$ be the permutative category with objects $(f,a)$ where $f\colon A\to X$ is a map in $\mathcal{A}^G_2$ and $a\in F(A)$.  The morphisms $(f,a)\to (g,b)$ consist of maps $F(f)(a)\to F(g)(b)$ in $F(X)$.  The permutative structure on $j(F)(X)$ is given by
    \[
        (A\xrightarrow{f} X,a)\oplus(B\xrightarrow{g} X,b) = (A\amalg B \xrightarrow{f\amalg g} X, (a,b))
    \]
    where $(a,b)\in F(A)\times F(B)\cong F(A\amalg B)$ uses the additivity of $F$.  This structure is strictly associative since we picked a model of $\mathcal{A}^G_{2}$ where disjoint union is a strictly associative biproduct. 

    If $q\colon X\to Y$ is a map in $\mathcal{A}^G_{2}$, define
    \begin{align*}
        j(F)(q)\colon j(F)(X) & \to j(F)(Y)\\
        (f\colon A\to X,a)& \mapsto (qf,a).
    \end{align*}
    If $r\colon Y\to Z$, we see that $j(rq) = j(r)j(q)$ because $\Perm$ and $\mathcal{A}^G_2$ are strict $2$-categories.  If $\beta\colon q\Rightarrow q'$ is a $2$-cell in $\mathcal{A}^G_2$ then the component \[j(F)(\beta)_{(f,a)}\colon (qf,a)\to (q'f,a)\] is given by the map $F(\beta\cdot f)\colon F(qf)(a)\to F(q'f)(a)$.
    Now suppose $\alpha\colon F\Rightarrow H$ is a pseudo-natural transformation and let $X\in \mathcal{A}^G_2$. We define a map $j(\alpha)_X\colon j(F)(X)\to j(H)(X)$ by 
    \[
        j(\alpha)_X(f\colon A\to X,a) = (f,\alpha_A(a))\in j(H)(X).
    \]
    Note that for any $q\colon X\to Y$ the diagram
    \[
        \begin{tikzcd}
        j(F)(X) \ar[d,"j(F)(q)"'] \ar[r,"j(\alpha)_X"]& j(H)(X) \ar[d,"j(H)(q)"] \\
        j(F)(Y) \ar[r, swap, "j(\alpha)_Y"]& j(H)(Y)
        \end{tikzcd}
    \]
    commutes because both composites take $(f,a)$ to $(qf,\alpha_A(a))$.  In sum, we have constructed a functor
    \[
        j\colon \Mack_{\ps}(\Perm_2)\to \Fun_{2}^{\mathrm{add}}(\mathcal{A}^G_2,\Perm_2).
    \]
\end{definition} 

\begin{lemma}\label{i and j nat 1}
    There is a natural transformation $\mu\colon i\circ j\Rightarrow 1$ which, objectwise, is composed of categorical equivalences.    
\end{lemma}
\begin{proof}
    For any additive pseudo-functor $F\colon \mathcal{A}^G_2\to \Perm_2$, define a pseudo-natural transformation $\mu_F\colon ij(F)\Rightarrow F$ by the component functors
    \begin{align*}
        (\mu_F)_X\colon ij(F)(X) &\to F(X) \\
        (f\colon A\to X,a) & \mapsto F(f)(a).
    \end{align*}
    Note that the morphisms in $ijF(X)$ are defined so that $(\mu_F)_X$ is a fully faithful functor. Moreover, since $F$ is strictly unital, taking the map $f$ to be an identity implies that the $(\mu_F)_X$ is surjective on objects.  Thus $(\mu_F)_X$ is a categorical equivalence for all $X$.
\end{proof}

\begin{lemma}\label{i and j nat 2}
    There is a natural transformation $\epsilon \colon j\circ i\Rightarrow 1$ whose components are strict $2$-natural transformations, each of whose component functors is a categorical equivalence.
\end{lemma}
\begin{proof}
The claim follows from the observation that the pseudo-natural transformation $\mu_F$ defined in the previous lemma is actually a strict natural transformation whenever $F$ is a strict $2$-functor. 
\end{proof}

   \begin{remark}\label{remark about strictification nonsense}
    One might have hoped to phrase \cref{i and j nat 1,i and j nat 2} as ``there are natural isomorphisms,'' but unfortunately this is not true since categorical equivalences are not all categorical isomorphisms.  Later, we will localize these categories at classes of morphisms which contain all categorical equivalences, and after these localizations the maps $i$ and $j$ do become categorical equivalences.
\end{remark}

\begin{corollary}\label{cor:compare ps and PC}
    There are functors \[
\begin{tikzcd}[column sep = large]
    \Mack_{\PC}(\Perm_{\PC}) & \Fun^{\rm add}_{2}(\mathcal{A}^G_{2},\Perm_{2}) \ar[l,hook',swap, "\Psi"]\ar[r, hook,shift right =2mm, swap, "i"] & \Mack_{\ps}(\Perm_{2}) \ar[l,shift right =2mm, swap, "j"]
\end{tikzcd}
    \] such that\begin{itemize}
        \item there is a natural transformation $\mu \colon i\circ j\Rightarrow 1$ whose components are pseudo-natural transformations, each of whose component functors are categorical equivalences, 
        \item there is a natural transformation $\epsilon \colon j\circ i\Rightarrow 1$ whose components are strict $2$-natural transformations, each of whose component functors are categorical equivalences, 
        \item $\Psi$ is the fully faithful functor of \cref{prop: strict to PC}.
    \end{itemize}
\end{corollary}

We end this section by showing how to replace $\Perm_2$ in \cref{cor:compare ps and PC} with a category of symmetric monoidal categories. Let $\Sym_2$ denote the $(2,1)$-category of symmetric monoidal categories, strictly unital strong monoidal functors, and natural isomorphisms.

\begin{definition}
    We say $F\colon \mathcal{A}^G_2\rightarrow \Sym_2$ is a \textit{symmetric monoidal Mackey functor} if it is an element of $\Mack_{\ps}(\Sym_2):=\Fun^{\mathrm{add}}_{\ps}(\mathcal{A}^G_2, \Sym_2)$.
\end{definition}

It is well-known, (see, e.g.\ \cite{Isbell}) that every symmetric monoidal category is categorically equivalent to a permutative category. A similar construction to \cref{defn: strictification functor} produces a pair of functors as in \cref{cor:compare ps and PC}.

\begin{lemma}\label{lem: compare perm and sym}
    There are functors \[
\begin{tikzcd}[column sep = large]
    \Mack_{\ps}(\Perm_{2}) \ar[r, hook,shift right =2mm, swap, "k"] & \Mack_{\ps}(\Sym_{2}) \ar[l,shift right =2mm, swap, "\ell"]
\end{tikzcd}
    \] such that\begin{itemize}
        \item there is a natural transformation $\mu \colon k\circ \ell\Rightarrow 1$ whose components are pseudo-natural transformations, each of whose component functors are categorical equivalences,
        \item there is a natural transformation $\epsilon \colon \ell\circ k\Rightarrow 1$ whose components are strict $2$-natural transformations, each of whose component functors are categorical equivalences.
    \end{itemize}
\end{lemma}

The reason for passing to symmetric monoidal categories is twofold. First, while permutative categories are convenient for proving theorems, actual examples of monoidal categories tend to be symmetric monoidal, and not permutative. For instance, the usual tensor product of abelian groups is symmetric monoidal but not permutative.  A second reason to consider these objects is that they provide a framework for understanding different kinds of monoidal structures which arise in equivariant stable homotopy theory.  

In \cite{HillHopkins}, the authors consider a version of symmetric monoidal Mackey functors which is the same in spirit as the definition we give here. At the level of algebra, their work provides a categorical interpretation of the difference between Tambara functors and Green functors \cite{HillMazur,Hoyer}.  In homotopy theory, these ideas give a way of understanding the various kinds of equivariant $E_{\infty}$-operads considered by Blumberg--Hill \cite{BlumbergHillOperadic}. 

\subsection{Mackey functors in \texorpdfstring{$\infty$}{TEXT}-categories}\label{sec:infty Mack}

Barwick and coauthors Glasman--Shah developed and studied an $\infty$-categorical approach to spectral Mackey functors \cite{BarwickOne,Barwick2}. A comparison between the $\infty$-categorical Mackey functors and genuine $G$-spectra can be found in \cite{Nardin, CMNN}. In this subsection, we review some of their definitions and results.

To define the $\infty$-categorical version of the Burnside category we first explain how to build an $\infty$-category out of a $(2,1)$-category.  Recall every ordinary category $\cat C$ determines an $\infty$-category $N_*(\cat C)$ via the nerve, and for $2$-categories there is a similar construction called the \emph{Duskin nerve}, denoted $N^D_*(\cat C)$ \cite{duskin}. When $\cat{C}$ is a strict $2$-category, the Duskin nerve of $\cat{C}$ is isomorphic (as a simplicial set) to the homotopy coherent nerve of the simplicial category obtained from $\cat{C}$ by applying the nerve to all the morphism categories \cite[Example 2.4.3.11]{kerodon}.

\begin{definition}
    The \textit{effective Burnside category} $\mathcal{A}^G_{\rm eff}$ is the Duskin nerve of $\mathcal{A}^G_2$.
\end{definition}

\begin{remark}
    We note that this definition of the effective Burnside category is not actually the one given in \cite{BarwickOne}.  As noted in \cite[Definition 2.1]{CMNN}, however, it is equivalent as an $\infty$-category to Barwick's effective Burnside category.
\end{remark}

Recall that an $\infty$-category $\cat C$ is said to be \textit{(pre-)additive} if and only if its homotopy category $h\cat C$ is (pre-)additive. Mackey functors can be valued in any preadditive $\infty$-category.

\begin{definition}\label{definition: infinity Mack functors}
    A \textit{Mackey functor} in a preadditive $\infty$-category $\cat C$ is a functor $\mathcal{A}^G_{\rm eff}\to \cat C$ which takes finite coproducts to products. The $\infty$-category of Mackey functors is denoted  $\Mack_\infty(\cat C)$.
    When $\cat C$ is the $\infty$-category $\smc$ of symmetric monoidal $\infty$-categories, we refer to such Mackey functors as symmetric monoidal $\infty$-Mackey functors. 
\end{definition}

If $\cat C$ is presentably symmetric monoidal, then $\Mack_\infty(\cat C)$ is also presentably symmetric monoidal, essentially via Day convolution (see \cite{Barwick2} or \cite[Construction 2.4]{CMNN}). As noted in \cite[Notation 6.3]{BarwickOne}, categories of Mackey functors are functorial in additive functors of $\infty$-categories. 

If $\cat C$ is an additive $1$-category, then $\Mack_{\infty}(N_*\cat C)\simeq N_*(\Mack(\cat C))$ (see \cite[Example 6.2]{BarwickOne}). This identification makes use of the fact that for $1$-categories $\cat C$ and $\cat D$, there is an equivalence of $\infty$-categories
\begin{equation}\label{eq: nerve iso}
\Fun_\infty(N_*(\cat C), N_*(\cat D)) \simeq N_*(\Fun(\cat C, \cat D)).
\end{equation}
We now establish a similar identification for $\infty$-categories of Mackey functors valued in Duskin nerves of $(2,1)$-categories, by proving that a version of \cref{eq: nerve iso} also holds for $(2,1)$-categories, where we replace the nerve with the Duskin nerve. 



Recall that the category of bicategories and strictly unital pseudo-functors forms a closed symmetric monoidal category, with internal hom $[\cat C, \cat D]_{\ps}$ given by the bicategory of (strictly unital) pseudo-functors, pseudo-natural transformations, and modifications (c.f. \cite{campbell_how_2019}); we use the notation $[\cat C, \cat D]_{\ps}$ to distinguish this bicategory from the $1$-category $\Fun_{\ps}(\cat C, \cat D)$. Observe that if $\cat C, \cat D$ are $(2,1)$-categories, then $[\cat C, \cat D]_{\ps}$ is also a (2,1)-category \cite[Corollary 4.4.13]{JohnsonYau2Cat}; in particular, all modifications are invertible as $\cat D$ is a $(2,1)$-category. 

\begin{proposition}\label{prop:dusk N and fun}
    Let $\cat C$ and $\cat D$ be $(2,1)$-categories. There is an equivalence of $\infty$-categories \[
    N^D_*\left( [\cat C, \cat D]_{\ps} \right) \simeq \Fun_{\infty}\left(N^D_*(\cat C), N^D_*(\cat D) \right).
    \]
\end{proposition}

Before proving \cref{prop:dusk N and fun}, we recall some necessary preliminaries, in particular that there is an $\infty$-categorical adjunction \begin{equation}\label{eq:Duskin N adj}
        \begin{tikzcd}
	{\Cat_{\infty}} && {\Cat_{(2,1)}}
	\arrow[""{name=0, anchor=center, inner sep=0}, "{h_2}", curve={height=-12pt}, from=1-1, to=1-3]
	\arrow[""{name=1, anchor=center, inner sep=0}, "{N^D_*}", curve={height=-12pt}, from=1-3, to=1-1]
	\arrow["\dashv"{anchor=center, rotate=-90}, draw=none, from=0, to=1]
\end{tikzcd}
\end{equation}
 where $h_2$ is a ``$(2,1)$-categorical truncation'' functor. The $\infty$-category $\Cat_{(2,1)}$ is formed as follows: there is a 2-category of $(2,1)$-categories whose hom categories are the groupoids of strictly unital pseudo-functors and pseudo-natural isomorphisms. The $\infty$-category $\Cat_{(2,1)}$ is obtained by first changing enrichment along the nerve functor and then applying homotopy coherent nerve. In particular, the Kan complex mapping space $\Map_{\Cat_{(2,1)}}(\cC, \cD)$ between $(2,1)$-categories $\cC$ and $\cD$ is the nerve of the groupoid of strictly unital pseudo-functors $\cC\to \cD$ and pseudo-natural isomorphisms. The adjunction above then follows from the composition of the $2$-categorical adjunctions \[\begin{tikzcd}
hQ\Cat && \Cat_{\rm Kan} && {\Cat_{(2,1)}}
	\arrow[""{name=0, anchor=center, inner sep=0}, "{C}", curve={height=-12pt}, from=1-1, to=1-3]
	\arrow[""{name=1, anchor=center, inner sep=0}, "{N^{hc}}", curve={height=-12pt}, from=1-3, to=1-1]
	\arrow["\dashv"{anchor=center, rotate=-90}, draw=none, from=0, to=1]
    	\arrow[""{name=0, anchor=center, inner sep=0}, "{h_\bullet}", curve={height=-12pt}, from=1-3, to=1-5]
	\arrow[""{name=1, anchor=center, inner sep=0}, "{N_\bullet}", curve={height=-12pt}, from=1-5, to=1-3]
	\arrow["\dashv"{anchor=center, rotate=-90}, draw=none, from=0, to=1]
\end{tikzcd}\]
where $hQ\Cat$ is the $2$-homotopy category of quasicategories (c.f. \cite[Example 2.1.5]{riehl_elements_2022}), $C\dashv N^{hc}$ is the homotopy coherent nerve, and $h_\bullet \dashv N_\bullet$ is the adjunction given by change of enrichment along the ordinary nerve \cite[Digression 1.4.2]{riehl_elements_2022}. 


\begin{remark}\label{rmk:Duskin is right adj}
We take note of a few key properties of the adjunction $h_2\dashv N_*^D$. As the Duskin nerve is fully-faithful (see \cite[\href{https://kerodon.net/tag/00AV}{Remark 00AV}]{kerodon}), for any $(2,1)$-category $\cC$ the counit of the adjunction induces an equivalence of $(2,1)$-categories $h_2N^D_*(\cC)\simeq \cC$. Moreover, one can check that $h_2$ preserves finite products, where products of $(2,1)$-categories are the usual Cartesian product.
\end{remark}

\begin{proof}[Proof of \cref{prop:dusk N and fun}]
    By the $\infty$-categorical Yoneda lemma, it suffices to show that there is an equivalence of mapping spaces\[
\Map_{\Cat_{\infty}}(\cat A, N^D_*([\cat C, \cat D]_{\ps}) \simeq \Map_{\Cat_{\infty}}(\cat A, \Fun_\infty(N^D_*\cat C, N^D_*\cat D))
\] for any $\infty$-category $\cat A$. 
By the adjunction of \cref{eq:Duskin N adj}, there is an equivalence of mapping spaces\[
\Map_{\Cat_{\infty}}(\cat A, N^D_*([\cat C, \cat D]_{\ps}) \simeq \Map_{\Cat_{(2,1)}}(h_2\cat A, [\cat C, \cat D]_{\ps}).
\] Hence it suffices to show that $\Map_{\Cat_{\infty}}(\cat A, \Fun_\infty(N^D_*\cat C, N^D_*\cat D))$ is equivalent to the latter mapping space. This follows from a series of equivalences\begin{align*}
    \Map_{\Cat_{\infty}}(\cat A, \Fun_{\infty}(N^D_*\cat C, N^D_*\cat D)) & \simeq \Map_{\Cat_{\infty}}(\cat A\times N^D_*\cat C, N^D_*\cat D))\\
    &\simeq \Map_{\Cat_{(2,1)}}(h_2(\cat A\times N^D_*\cat C), N_*^D\cat D)\\
    &\simeq \Map_{\Cat_{(2,1)}}(h_2\cat A\times \cat C, N_*^D\cat D)\\
    &\simeq \Map_{\Cat_{(2,1)}}(h_2A, [\cat C, \cat D]_{\ps}),
\end{align*}
using the adjunction of \cref{eq:Duskin N adj}, the closed monoidal structure of both $\Cat_{\infty}$ and $\Cat_{(2,1)}$, and properties of $h_2$ as noted in \cref{rmk:Duskin is right adj}.
\end{proof}

\begin{remark}
More generally, if $\cat C$ and $\cat D$ are $2$-categories, there is a bijection between strictly unital lax functors $\cat C\to \cat D$ and morphisms of simplicial sets $N^D_*\cat C\to N^D_*\cat D$. However, the Duskin nerve produces an $\infty$-category if and only if the input is a $(2,1)$-category \cite{duskin}, in which case the notions of lax and pseudo-functors agree (as noted in \cref{rem:lax versus pseudo}). 
\end{remark}

For the purposes of Mackey functors, we also want to restrict to subcategories of additive functors.

\begin{lemma}\label{lem:dusk N and additive fun}
    If $\cat C$ and $\cat D$ are additive $(2,1)$-categories, there is an equivalence of $\infty$-categories\[
    N_*^D\left( [\cat C, \cat D]_{\ps}^{\rm add} \right) \simeq \Fun_{\infty}^{\rm add}\left(N^D_*(\cat C), N^D_*(\cat D) \right),
    \]
\end{lemma}\begin{proof}
The additive condition is just a condition on $0$-cells, and it is straightforward to check that the equivalence described in \cite{duskin} (see also \cite[\href{https://kerodon.net/tag/00AU}{Tag 00AU}]{kerodon}) sends additive functors to additive functors.
\end{proof}

We now prove a variant of the above which takes into account localizations.  Recall that if $W$ is a class of morphisms in an $\infty$-category $\cat A$, then there is an associated localization $\cat A[W^{-1}]$ and a functor $\cat A\to \cat A[W^{-1}]$, universal among those functors which send all morphisms in $W$ to equivalences \cite[1.3.4.1]{HA}. 

\begin{remark}
    Suppose that $\cat{D}$ is a $(2,1)$-category and $W$ is a collection of $1$-cells.  The morphisms in the $\infty$-category $N^D_*(\cat{D})$ are the same as the $1$-cells in $\cat{D}$, and so the localization $N^D_*(\cat{D})[W^{-1}]$ makes sense.
\end{remark}

\begin{notation}
    Suppose that $\cat{C}$ and $\cat{D}$ are $\infty$-categories and $W\subset \cat{D}$ is some collection of morphisms. We will write $W_{\lvl}$ for the collection of morphisms in $\Fun_{\infty}(\cat{C},\cat{D})$ for the collection of natural transformations $\alpha\colon F\Rightarrow G$ such that for all $x\in \cat{C}$ the component $\alpha_x\colon F(x)\to G(x)$ is in $W$.  
\end{notation}

\begin{proposition}\label{lem:fun and nerve of DK localization}
Let $\cat C$ and $\cat D$ be $(2,1)$-categories and let $(N_*^D(\cat D))[W^{-1}]$ be the localization of the Duskin nerve of $\cat D$ with respect to a class of weak equivalences $W\subset \cat D$. There is an equivalence of $\infty$-categories
\[
N_*^D\left( [\cat C, \cat D]_{\ps} \right) \left[ W_{\lvl}^{-1}\right]\simeq \Fun_{\infty}\left(N^D_*(\cat C), \left( N^D_*(\cat D) \right)\left[W^{-1}\right] \right).
\]
\end{proposition}
\begin{proof}
  By \cref{prop:dusk N and fun}, we have
  \[
  N_*^D\left( [\cat C, \cat D ]_{\ps}\right) \simeq \Fun_{\infty}\left(N^D_*(\cat C), N^D_*(\cat D)  \right).
  \]
The levelwise equivalences on the left side correspond to the equivalences $W$ in $N^D_*(\cat D)$ on the right side.  
Let $\cat D'$ be any $\infty$-category. 
Any functor 
\[
\Fun_{\infty}\left(N^D_*(\cat C), \left(N^D_*(\cat D)\right) \right) \longrightarrow \cat D'
\]
that sends $W$ to equivalences in $\cat D'$ must factor through $\Fun_{\infty}\left(N^D_*(\cat C), \left(N^D_*(\cat D)\right) \left[ W^{-1}\right]\right)$ via the localization $N^D_*(\cat D)\rightarrow N^D_*(\cat D)[W^{-1}]$.
Therefore by the universal property of the Dwyer--Kan localization \cite[1.3.4.1]{HA}, we obtain the desired equivalence. 
\end{proof}

Restricting to additive functors and repeating this argument yields the following corollary.

\begin{corollary}\label{cor:additive fun and nerve of DK localization}
Let $\cat C$ and $\cat D$ be additive $(2,1)$-categories and $W$ a class of weak equivalences in $\cat D$. There is an equivalence of $\infty$-categories\[
N_*^D\left( [\cat C, \cat D]_{\ps}^{\rm add} \right) \left[ W_{\lvl}^{-1}\right]\simeq \Fun^{\rm add}_{\infty}\left(N^D_*(\cat C), \left( N^D_*(\cat D) \right)\left[W^{-1}\right] \right).
\]    
\end{corollary}

\begin{example}\label{ex:Mack of localization}
    Let $\cat C$ be an additive $(2,1)$-category and $N_*^D(\cat C)[W^{-1}]$ be the localization of its Duskin nerve with respect to a class of weak equivalences $W$. By \cref{lem:dusk N and additive fun}, if $\cat C$ is a preadditive $(2,1)$-category, then $\Mack_\infty(N_*^D(\cat C))$ is equivalent to the Duskin nerve of $[\mathcal{A}_2^G, \cat C]_{\ps}^{\rm add}$, where the objects are additive pseudo-functors, morphisms are pseudo-natural transformations, and 2-morphisms are modifications. Moreover, there is an equivalence of $\infty$-categories\[
    N_*^D([\mathcal{A}_2^G, \cat C]_{\ps}^{\rm add})[W^{-1}_{\lvl}] \simeq \Mack_\infty(N_*^D(\cat C)[W^{-1}]).
    \]
\end{example}
\begin{example}\label{ex:Mack of 1cat localization}
    Suppose that $\cat C$ is an additive $1$-category, thought of as a $2$-category with only identity $2$-cells. Then $[\mathcal{A}_2^G, \cat{C}]_{\ps}^{\rm add}$ is a $2$-category with only identity $2$-cells since the structure data of a modification in $[\mathcal{A}_2^G,\cat{C}]_{\ps}^{\rm add}$ consists of $2$-cells in $\cat{C}$ which are all identities.  Thus the Duskin nerves in \cref{ex:Mack of localization} are ordinary nerves and we obtain
    \[
      N _*(\Mack(\cat C))[W^{-1}_{\rm lvl}] \simeq \Mack_\infty(N_*(\cat C)[W^{-1}]).
    \] 
\end{example}

We will make use of this example in the case when $\cat C$ is $\Perm$ or $\Perm_2$, and the weak equivalences are $W_K$, the class of stable equivalences (\cref{defn:st equiv in perm}).

\begin{corollary}\label{lem:Mack ps Perm is infty Mack} There are equivalences of $\infty$-categories of Mackey functors
    \[\Mack_{\infty}(N_*(\Perm)[W_K^{-1}]) \simeq \Mack_\infty(N_*(\Sp^{\geq 0})[W^{-1}_{\lvl}])\]
    and hence in particular\[
    N_*(\Mack(\Perm))[(W_{{K}})_{\rm lvl}^{-1}]\simeq \Mack_\infty(N_*(\Sp^{\geq 0})[W^{-1}_{\lvl}]).
    \]
\end{corollary}
\begin{proof}
    This result is an immediate consequence of \cref{thm:Mandell for infty cats} together with \cref{ex:Mack of 1cat localization}.
\end{proof}

\subsection{Spectral Mackey functors}\label{sec:Sp Mack}
Spectral Mackey functors can be thought of as diagrams of ordinary spectra indexed by the subgroups of $G$ and connected by transfer and restriction maps. This definition can be made precise in both the $\infty$-categorical and enriched setting. In $\infty$-categories, one simply considers $\Mack_{\infty}(\Sp)$ and there is an equivalence of $\infty$-categories $\Mack_{\infty}(\Sp)\simeq \Sp_{\infty}^G$ \cite{BarwickOne} (see also \cite[Appendix A]{CMNN}).

In the enriched setting, spectral Mackey functors are constructed as $\Sp$-enriched functors between certain $\Sp$-categories.
Since $\mathcal{A}^G_{\PC}$ is enriched in permutative categories, we may take $K$-theory of the hom-categories to obtain a spectrally enriched version of the Burnside category, which we denote $\mathcal{A}^G_{\rm Sp}$. Following \cite{GuillouMay:2011}, we change enrichment along the (non-symmetric) $K$-theory multifunctor from \cite{GMMO-final}. 

\begin{definition}
    A \textit{spectral Mackey functor} is a spectrally enriched functor $\mathcal{A}^G_{\rm Sp}\to {\rm Sp}$.  Here, $\Sp$ denotes the the closed monoidal category of orthogonal spectra, naturally enriched over itself. We denote the category of spectral Mackey functors by $\Mack_{\Sp}(\Sp):=\Fun_{\rm Sp}(\mathcal{A}^G_{\rm Sp}, \Sp)$. 
\end{definition}

In \cite{GuillouMay:2011}, Guillou--May show that there is a zig-zag of Quillen equivalences between orthogonal $G$-spectra and $\Fun_{\Sp}((\mathcal{A}^G_{\Sp})^{\op}, \Sp)$, a category of \textit{contravariant} Mackey functors. While Guillou--May use contravariant functors in their definition of Mackey functors, other authors (e.g. \cite{BarwickOne}) opt for covariant functors. Morally, the resulting definition should be the same because spans can always be ``turned around,'' and in most settings the Burnside category is indeed self-dual. Although it is not known to be true in the enriched setting (as discussed further in \cref{sec:main}), we show that there is no meaningful difference for the purposes of Mackey functors. In particular, we establish a ``no op'' version of the Guillou--May theorem; we defer the proof to \cref{app:GM thm}, in \cpageref{proof of covariant guillou-may}, as the details are quite technical and not essential to understanding the rest of the paper.

\begin{theorem}\label{thm:GM no op}
    There is a zig-zag of Quillen equivalences between the model categories $\Sp^G$ with the stable model structure and $\Fun_{\Sp}(\mathcal{A}^G_{\Sp},\Sp)$, where the latter is given the projective level model structure.
\end{theorem}


As we shall see in the next section, the spectral Mackey functor approach to genuine $G$-spectra is particularly well-suited to constructions involving $K$-theory. In particular, as discussed in the introduction, the covariant Guillou--May theorem above will resolve some technical complications that arise in Bohmann--Osorno's construction \cite{BohmannOsorno:2015}. 

\section{Bohmann--Osorno's \texorpdfstring{$K$}{TEXT}-theory machine, revisited}\label{sec:main}

In \cite{BohmannOsorno:2015}, Bohmann--Osorno construct a functor  \[\mathbb{K}^{BO}\colon \Fun^{\times}_{\PC}((\mathcal A^{\op})_{\PC}, \Perm_{\PC})\to \Fun^{\times}_{\Sp}(K^{\GMMO}_{\bullet}(\mathcal A_{\PC}^{\op}), \Sp)\] by first doing a change of enrichment along the $K$-theory functor $K^{\GMMO}$ from \cite{GMMO-final} and then post-composing with a spectrally enriched version of $K^{\GMMO}$. However, as discussed in \cite[Remark 10.5.5]{johnsonYau:homotopyTheoryOfMackey}, it is unclear whether the codomain of $\mathbb{K}^{\rm BO}$ is Quillen equivalent to the category of genuine orthogonal $G$-spectra. This is because the Guillou--May theorem \cite{GuillouMay:2011} establishes an equivalence\[
\Fun^{\times}_{\Sp}((K_{\bullet}^{\GMMO}(\mathcal{A}_{\PC}))^{\op}, \Sp)\simeq \Sp^G,
\] and duality (taking opposite categories) does not necessarily commute with a change of enrichment along a \textit{non-symmetric} multifunctor, such as $K^{\GMMO}$. 

\begin{remark}
    One might attempt to fix this issue by replacing $K^{\GMMO}$ in Bohmann-Osorno's construction with $K^{\mathrm{EM}}$, the Elmendorf--Mandell $K$-theory functor from \cite{Elmendorf-Mandell:06}, which \textit{is} a symmetric multifunctor. However, to conclude this functor lands in $G$-spectra, one would still need to prove a version of the Guillou--May theorem for $\Fun^{\times}_{\Sp}((K_{\bullet}^{\mathrm{EM}}(\mathcal{A}_{\PC}))^{\op}, \Sp)$. See \cite[Section 10.5]{johnsonYau:homotopyTheoryOfMackey} for details and more discussion.
\end{remark}

The covariant Guillou--May theorem (\cref{thm:GM no op}) allows us to avoid the subtleties surrounding duality and change of enrichment, and we can easily adapt Bohmann--Osorno's construction to an ``un-opped'' version that builds spectral Mackey functors from $\PC$-Mackey functors.

\begin{theorem}[cf. \cite{BohmannOsorno:2015}]\label{thm:BO}
    There is a $K$-theory functor
    \[
        \mathbb{K}\colon\Mack_{\PC}(\Perm_{\PC})\to  \Mack_{\Sp}(\Sp).
    \]
\end{theorem}
\begin{proof}
The argument is the same as in \cite{BohmannOsorno:2015}. First, change enrichment along the multifunctor $K^{\GMMO}\colon \PC\to \Sp$, from $\PC$-enriched categories to spectrally enriched categories. In particular, we can do a change of enrichment to get a spectrally enriched category $\Perm_{\Sp}$, which has the same objects as $\Perm_{\PC}$ but morphism spectra $K^{\GMMO}(\Perm(\cat A, \cat B))$. Thus, there is a functor\[
\Mack_{\PC}({\Perm}_{\PC}) \xrightarrow{K_{\bullet}} \Mack_{\Sp}({\Perm}_{\Sp})
\] by changing enrichment on both $\mathcal{A}^G$ and $\Perm$.
In \cite[Theorem 6.2]{BohmannOsorno:2015}, Bohmann--Osorno show that $K$-theory defines a spectrally enriched functor $\Phi\colon \Perm_{\Sp}\to \Sp$, and so every $\PC$-Mackey functor determines a spectral Mackey functor via\[
    \begin{tikzcd}
        \Mack_{\PC}({\Perm}_{\PC})\ar[r, "K_{\bullet}"]\ar[rd, dashed, swap, "\mathbb{K}"] & \Mack_{\Sp}({\Perm}_{\Sp})\ar[d, "\Phi"] \\
        & \Mack_{\Sp}(\Sp). 
    \end{tikzcd}
    \]\qedhere
\end{proof}

Our previous discussion (particularly \cref{cor:compare ps and PC,lem: compare perm and sym}) gives us a way to extend Bohmann--Osorno's construction to both symmetric monoidal and permutative Mackey functors.

\begin{definition}\label{defn:KT of sm and perm Macks}
    Define the $K$-theory of symmetric monoidal and permutative Mackey functors by the composites 
    \[
\begin{tikzcd}
     \Mack_{\ps}(\Sym_2) \ar[r,"\ell"]   \ar[drr,swap,dashed,"\mathbb{K}"]& \Mack_{\ps}(\Perm_2) \ar[rd, dashed,"\mathbb{K}"] \ar[r, "\Psi\circ j"] &\Mack_{\PC}(\Perm_{\PC}) \ar[d, "\mathbb{K}"] \\
   & &\Mack_{\Sp}(\Sp).
\end{tikzcd}
    \]
\end{definition}

Passing through the ziag-zag of Quillen equivalences
$\Mack_{\Sp}(\Sp)\leftrightarrow \Sp^G$
of \cref{thm:GM no op} exhibits the $K$-theory of permutative Mackey functors as a genuine $G$-spectrum. We now consider some examples.

\begin{example}\label{ex:G-BPQ}
    For any finite $G$-set $X$, let $A_X$ denote the co-representable pseudo-functor $A_X(Y) = \mathcal{A}^G_2(X,Y)$.  We observe in \cref{remark: co-representable functors} that the $\mathbb{K}(A_X)$ corresponds, under the zig-zag of Quillen equivalences in \cref{thm:GM no op}, to the dual of the suspension $G$-spectrum $\mathbb{X} = \Sigma^{\infty}_G(X_+)$.  In particular, in the equivariant stable homotopy category $\mathbb{K}(A_X)$ is isomorphic to the dual of $\mathbb{X}$, which is also isomorphic to $\mathbb{X}$ by \cref{lemma: alpha hat is weak equiv}.
\end{example}

\begin{example}\label{ex:G-susp}
    As a generalization of the previous example, now let $X$ be a pointed $G$-space. In \cite[\S 3]{malkiewich/merling:22}, Malkiewich--Merling construct the genuine suspension spectrum $\Sigma^\infty_G X$ as the equivariant $K$-theory of a certain symmetric monoidal category of ``homotopy discrete'' retractive $G$-spaces. Their construction essentially arises from a pseudo-functor $\mathcal{A}^G_2\to {\rm Wald}$, where the codomain is the category of small Waldhausen categories. They show in Section 3.3 (see also Section 2.3) that this construction factors through symmetric monoidal categories, and applying $\mathbb{K}$ to the resulting symmetric monoidal Mackey functor produces $\Sigma^\infty_G X$.
\end{example}

\begin{example}\label{ex:coefficient ring K theory}
    
    In \cite{CalleChanMejia:Linearization}, the first and second-named authors, together with Mejia, show that there is a $K$-theory construction for \emph{coefficient rings}, an equivariant generalization of rings, that outputs a genuine $G$-spectrum. This construction comes from a pseudo-functor $\mathcal{A}^G_2\to \mathrm{Wald}$, but in fact factors through symmetric monoidal categories. By the results of \cite{BohmannOsorno2020}, applying $\mathbb{K}$ to the resulting symmetric monoidal Mackey functor recovers the $K$-theory of coefficients rings from \cite{CalleChanMejia:Linearization}. A similar construction produces equivariant $K$-theory spectra associated to Green functors, Tambara functors, or even ring $G$-spectra.  
\end{example}

\begin{example}\label{ex:ordinary Mackey functor}
    Let $M\colon \mathcal{A}^G\to \mathrm{Ab}$ be an ordinary Mackey functor.  Viewing $\mathcal{A}^G$ as a $2$-category with only identity morphisms, there is a pseudo-functor $\mathcal{A}^G_2\to \mathcal{A}^G$; pulling back along this functor allows us to view $M$ as a pseudo-functor $\mathcal{A}^G_2\to \mathrm{Ab}$. The $G$-spectrum $\mathbb{K}(A)$ is the Eilenberg--Mac Lane spectrum $\mathrm{H}A$ by a covariant version of \cite[Theorem 8.1]{BohmannOsorno:2015}.
\end{example}


One could hope to leverage Mandell's work to show that the $K$-theory functor of \cref{thm:BO} models all connective genuine $G$-spectra. The immediately appealing way to prove this result would be to apply Mandell's construction levelwise, however, this approach would require the inverse $K$-theory construction to be a multifunctor (to do a change of enrichment) and either spectrally enriched or a $\PC$-functor (to post-compose). 

Mandell's inverse $K$-theory functor factors as\[
K^{-1}\colon \Gamma\text{-sSet} \xrightarrow{S_*} \Gamma\text{-Cat} \xrightarrow{P} \Perm. 
\] where $P$ is a certain Grothendieck construction described in detail in \cite[\S 4]{Mandell/invKT} (see also \cite[\S 5]{JohnsonYau/invKT}) and $S_*$ is induced by the levelwise application the left adjoint of the Quillen equivalence between the Thomason model structure on Cat and the classical model structure on simplicial sets.
Although the functor $P$ is sufficiently multifunctorial, as shown in \cite{JohnsonYau/invKT} (see also \cite{Elmendorf2021}), the entire $K^{-1}$ functor is not, as explained in \cite[\S 1]{JohnsonYau/invKT}. 

By using an $\infty$-categorical framework, we can bypass these issues of multifunctoriality. Both $N_*(\Perm)$ and $\Sp_\infty$ are additive $\infty$-categories since their homotopy categories are canonically equivalent to the homotopy category of connected spectra, which is additive.
Recall that $W_K$ is the class of stable equivalences in $\Perm$ (\cref{defn:st equiv in perm}). The additive structure on $N_*(\Perm)[W_K^{-1}]$ is defined precisely to make $K$ and $K^{-1}$ additive.

\begin{lemma}
The functor $K\colon N_*(\Perm)[W_K^{-1}]\rightarrow \Sp^{\geq 0}_\infty$ and its inverse $K^{-1}$ are additive functors of $\infty$-categories.  That is, they preserve products up to weak equivalence.
\end{lemma}
\begin{proof}
    Products in these $\infty$-categories can be computed, up to weak equivalence, by products in the categories $\Perm$ and $\Gamma\text{-sSet}$ by \cite[1.3.4.23]{HA}.  These categories are pre-additive, and so product and coproducts agree and it suffices to show $K$ and $K^{-1}$ preserve either products or coproducts.  This is immediate since $K$ and $K^{-1}$ participate in an adjunction (c.f. \cite{Mandell/invKT}).
\end{proof}

By functoriality of $\Mack_\infty(-)$ (see \cref{definition: infinity Mack functors})  and the equivalence $\Mack_\infty(\Sp_\infty)\simeq \Sp_\infty^G$ (c.f. \cite[Appendix A]{CMNN}), we obtain an equivariant $\infty$-categorical version of Thomason's result. Recall that if $W$ is a class of weak equivalences, then $W_{\rm lvl}$ denotes the morphisms of Mackey functors which are valued in $W$ levelwise.

\begin{corollary}\label{cor:Mack Mandell}
    By \cref{thm:Mandell for infty cats}, the functors $K$ and $K^{-1}$ induce an equivalence of $\infty$-categories
    \[
        \begin{tikzcd}
         \Mack_\infty(N_*({\Perm}))[(W_K)_{\rm lvl}^{-1}] \ar[shift left=2]{r}{\Mack_{\infty}(K)}[swap]{\simeq}    & [2em]\Mack_{\infty}(\Sp^{\geq 0}_{\infty}) \simeq (\Sp^{G}_{\infty})^{\geq 0} \ar[l,shift left=2, "\Mack_{\infty}(K^{-1})"]
        \end{tikzcd}
    \]
    and moreover \[\Mack_\infty(N_*({\Sym}))[(W_K)_{\rm lvl}^{-1}]\simeq \Mack_\infty(N_*({\Perm}))[(W_K)_{\rm lvl}^{-1}].\]
    Similarly, by \cref{cor:equiv smc and sp}, there is an equivalence of $\infty$-categories
\[
        \begin{tikzcd}[column sep = large]
        \Mack_\infty(\smc)[(W_\infty)^{-1}_{\rm lvl}] \ar[shift left=2]{r}[swap]{\simeq} & (\Sp^{G}_{\infty})^{\geq 0} \ar[l,shift left=2]
        \end{tikzcd}.
    \]  
\end{corollary}

Our main theorem is that the equivalence of $\infty$-categories from \cref{cor:Mack Mandell} is suitably compatible with Bohmann--Osorno's $K$-theory machine. We use $W$ to denote stable equivalences and note that $\mathbb{K}((W_{{K}})_{\rm lvl})\subseteq W_{\mathrm{lvl}}$ by definition.

   \begin{theorem}\label{thm:main}
    The following diagram of $\infty$-categories commutes:
    \[
    \begin{tikzcd}
    N_*(\Mack_{\PC}(\Perm_{\PC}))[(W_K)_{\rm lvl}^{-1}] \ar[r, "N_*\mathbb{K}"] & N_*(\Mack_{\Sp}(\Sp))[W_{\mathrm{lvl}}^{-1}] \ar[dd, swap,"\sim"']  \\\
 N_*(\Mack_{\ps}(\Perm_2))[(W_K)_{\rm lvl}^{-1}] \ar[u, "\Psi\circ j"] &    \\
N_*(\Mack(\Perm))[(W_K)_{\rm lvl}^{-1}] \ar[u, "\iota"] \ar[d, swap, "\sim"] & \Sp^G_\infty\\
\Mack_\infty(N_*(\Perm)[W_K^{-1}]) \ar[r, swap, "\Mack_\infty(K)"] & \Mack_\infty(\Sp_\infty) \ar[u, swap, "\sim"]
    \end{tikzcd},
    \]  where 
    \begin{enumerate}
        \item the functor $\Psi \circ j$ is map induced on localizations by the composition of the functors of \cref{cor:compare ps and PC} and \cref{prop: strict to PC};
        \item the functor $\iota$ is induced by the inclusion $\Mack(\Perm)\to \Mack_{\ps}(\Perm_2)$;
        \item the functors $\mathbb{K}$ and $\Mack_{\infty}(K)$ are the $K$-theory functors of  \cref{thm:BO} and \cref{cor:Mack Mandell}, respectively; and
        \item the left vertical arrow labeled $\sim$ is the equivalence of \cref{ex:Mack of 1cat localization},
        \item the right vertical arrows are the equivalences of $\infty$-categories from \cref{thm:GM no op} and \cite[Theorem A.1]{CMNN}.
    \end{enumerate}
\end{theorem}
\begin{proof}
The key point of the argument is a comparison between the enriched and $\infty$-categorical constructions of spectral Mackey functors. To reduce to this comparison, we observe that the bottom row of the diagram factors as \[
\begin{tikzpicture}[baseline= (a).base]
\node[scale=0.95] (a) at (1,1){
\begin{tikzcd}
    N_*\Mack(\Perm)[(W_K)_{\rm lvl}^{-1}] \ar[r, "\sim"] & \Mack_\infty(N_*(\Perm)[W_K^{-1}]) \ar[dr,swap, "\Mack_{\infty}(K)"] \ar[r, "\sim"] & \Mack_\infty(N_*(\Sp^{\geq 0})[W^{-1}]) \ar[d, hook] \\
    && \Mack_{\infty}(\Sp_{\infty}),
    \end{tikzcd}
    };
    \end{tikzpicture}
    \] where the second equivalence is from \cref{lem:Mack ps Perm is infty Mack}. By the universal property of the localization, there is a functor\[
    \Phi\colon \Mack_\infty(N_*(\Sp^{\geq 0})[W^{-1}])\to N_*\Mack_{\Sp}(\Sp^{\geq 0})[W^{-1}_{\rm lvl}]
    \] that fits into the original diagram as
    \[
    \begin{tikzpicture}[baseline= (a).base]
\node[scale=0.90] (a) at (1,1){
    \begin{tikzcd}
 N_*(\Mack_{\PC}(\Perm_{\PC}))[(W_K)_{\rm lvl}^{-1}] \ar[r, "N_*\mathbb{K}"] & N_*(\Mack_{\Sp}(\Sp^{\geq 0}))[W^{-1}_{\rm lvl}]\ar[r] & N_*(\Mack_{\Sp}(\Sp))[W_{\mathrm{lvl}}^{-1}] \ar[d, swap,"\sim"'] \\
N_*(\Mack(\Perm))[(W_K)_{\rm lvl}^{-1}] \ar[u, "\Psi\circ j\circ \iota"]  \ar[d, "\sim"'] && \Sp^G_\infty\\
\Mack_\infty(N_*(\Perm)[W_K^{-1}]) \ar[r, swap, "\Mack_\infty(K)"] & \Mack_{\infty}(N_*(\Sp^{\geq 0})[W^{-1})) \ar[uu, dashed, "\Phi"] \ar[r, hook] & \Mack_\infty(\Sp_\infty), \ar[u, swap, "\sim"]
    \end{tikzcd}
    };
    \end{tikzpicture}
    \] so that the left part of the diagram commutes; hence it suffices to show that the right part commutes.
    Using the definitions of all functors involved, we compute that on objects, $\Phi$ takes a functor 
    \[
        F\colon \mathcal{A}^G_{\rm eff}\to N_*(\Sp^{\geq 0})[W^{-1}]
    \]
    to a functor $\Phi(F)$ given by
    \[
        \Phi(F)(X) = K(j(K^{-1}(F(X))))
    \]
    where $j$ is a replacement of the permutative category $K^{-1}(F(X))$ which is naturally equivalent to $F(X)$, as a permutative category.  Thus on objects $\Phi(F)$ and $F$ take values in spectra which are equivalent levelwise, and these equivalences are natural.  Since colimits in functor categories are computed levelwise, it follows that $\Phi$ preserves colimits.

    It suffices to show that the diagram
    \[
        \begin{tikzcd}
            \Mack_\infty(N_*(\Sp^{\geq 0})[W^{-1}])  \ar[r,"\Phi"] \ar[d] & N_*\Mack_{\Sp}(\Sp^{\geq 0})[W_{\rm lvl}^{-1}] \ar[d] \\
            \Mack_\infty(N_*(\Sp)[W^{-1}]) \ar[dr, swap, "L^{-1}"] & N_*\Mack_{\Sp}(\Sp)[W_{\rm lvl}^{-1}] \ar[d,"GM"] \\
             & \Sp^G_{\infty}
        \end{tikzcd}
    \] commutes,
    where $GM$ is the equivalence of categories from \cref{thm:GM no op}, and $L^{-1}$ is an inverse to the functor induced on localizations by \cite[Theorem A.1]{CMNN}.  Every arrow in the above diagram preserves colimits, and so the diagram will commute, up to natural isomorphism, as long as we can show that the two ways of going around compute the same value, up to natural isomorphism, on the representable functors $A_{X}(Y) = N_*^D(\mathcal{A}^G_2)(X,Y)$ which generate the top-left category under colimits. To see that representables generate, note that by \cref{cor:Mack Mandell} the top-left category is equivalent to the $\infty$-category of connective $G$-spectra, which are generated under colimits by the suspensions of finite $G$-sets.  Under the equivalence of \cite[Theorem A.1]{CMNN}, the suspensions of finite $G$-sets correspond to the representable functors.
    
    By the two relevant formulations of the equivariant Barratt--Priddy--Quillen theorem, \cite[Theorem A.1]{CMNN} and  \cref{ex:G-BPQ}, (see also \cite[Theorem 9.4]{BohmannOsorno:2015}), the two $K$-theory constructions send $A_X$ to the spectral Mackey functors representing $\Sigma^{\infty}_{G}(X)$, the equivariant suspension spectrum of $X$, and its dual, respectively.  Since these objects are naturally isomorphic in the homotopy category, by \cref{proposition: wirthmuller iso}, these composite agree up to a natural equivalence on the representable functors and thus the diagram of $\infty$-categories commutes.
    \end{proof}

One might hope that the functors $\iota$, $j$, and $\Psi$ in \cref{thm:main} induce equivalences on homotopy categories, but this is only clear for $j$ (\cref{cor:compare ps and PC}, \cref{rmk:psi not ess surj}). However, since $\iota$ and $\Psi$ are at least fully faithful, we still obtain the desired result when restricting $\mathbb{K}$ to the image of $\Psi\circ j\circ \iota$ in $\PC$-Mackey functors.

\begin{corollary}\label{cor:str becomes equiv after loc}
    The functor $\mathbb{K}$ induces equivalences of categories
    \[\Mack(\Sym)[(W_{{K}})^{-1}_{\rm lvl}]\simeq \Mack(\Perm)[(W_{{K}})_{\rm lvl}^{-1}]\simeq h\Sp^G_{\geq 0}. \]
\end{corollary}\begin{proof}
    We view $\Sym\subseteq \Sym_2$ and $\Perm\subseteq \Perm_2$ as the inclusions of subcategories with only identity $2$-cells. These inclusions induce fully faithful inclusions $\Mack(\Sym)\hookrightarrow \Mack_{\ps}(\Sym_2)$ and $\Mack(\Perm)\hookrightarrow \Mack_{\ps}(\Perm_2)$ compatible with the definition of $\mathbb{K}$ (\cref{defn:KT of sm and perm Macks}). By \cref{thm:main}, we obtain a commutative diagram\[
    \begin{tikzcd}
        N_*(\Mack(\Sym))[W_{\rm lvl}^{-1}] \ar[r, "N_*\ell"]\ar[d,swap,  "\sim"] & N_*(\Mack(\Perm))[W_{\rm lvl}^{-1}] \ar[r, "N_*\mathbb{K}"] \ar[d, "\sim"] & \Sp_\infty^G
        \\
        \Mack_{\infty}(N_*(\Sym)[W_K^{-1}]) \ar[r, swap, "\sim"] & \Mack_{\infty}(N_*(\Perm)[W_K^{-1}]) \ar[ur, swap, "\Mack_{\infty}(K)"] &  
    \end{tikzcd}
    \] where the vertical arrows are equivalences by \cref{ex:Mack of 1cat localization}, $\ell$ is the strictification functor from \cref{lem: compare perm and sym} which becomes an equivalence after localization, and the bottom horizontal arrow is an equivalence by \cref{cor:Mack Mandell}. Finally, since $\Mack_\infty(K)$ induces an equivalence on the subcategory of connective $G$-spectra, $N_*\mathbb{K}$ does as well. The result then follows by taking homotopy categories.
\end{proof}

In particular, as a corollary, we obtain the desired result that Bohmann-Osorno's $K$-theory machine is essentially surjective.

\begin{corollary}\label{cor:BO is all}
For any connective $G$-spectrum $E$, there exists a symmetric monoidal Mackey functor $F\in \Mack_{\ps}(\Sym)$ so that $\mathbb{K}(F)$ is isomorphic to $E$ in $h\Sp^G_{\geq 0}$. Consequently, the same statement holds with $F\in \Mack_{\PC}(\Perm_{\PC})$.
\end{corollary}\begin{proof}
    The first statement follows directly from \cref{cor:str becomes equiv after loc}. As the definition of $\mathbb{K}(F)$ first strictifies the symmetric monoidal Mackey functor $F$ to a $\PC$-Mackey functor, this implies the second statement.
\end{proof}





A reasonable question to ask is whether this equivariant $K$-theory machine preserves multiplicative structure. 
Although the functor $K\colon \smc\rightarrow \Sp^{\geq 0}_\infty$ is lax symmetric monoidal, the equivariant case is more subtle.
It is true that the induced functor
\[
\Mack_\infty(\smc)\rightarrow \Mack_\infty(\Sp^{\geq 0})
\]
remains lax symmetric monoidal, however, the equivalence
$
\Mack_\infty(\Sp)\simeq \Sp^G
$ 
is \emph{not} lax symmetric monoidal, which we can see using the perspective of pseudo-functors.
The category of pseudo-functors admits a monoidal structure, given by Day convolution. If $K$-theory preserved multiplicative structure, it would send commutative monoids with respect to this structure to equivariant ring spectra. However, as discussed in the next remark, this is not the case.

\begin{remark}\label{remark: green not tambara}
The inclusion of ordinary Mackey functors into permutative Mackey functors preserves Day convolution monoids.  A monoid with respect to the Day convolution in Mackey functors is known as a Green functor, and by \cref{ex:ordinary Mackey functor}, every Green functor $A$ is realized as $\pi_0(\mathbb{K}(A))$.  Work of Brun \cite{Brun} shows that only a certain class of Green functors, known as Tambara functors, can appear as $\pi_0$ of genuine equivariant ring spectra.  Since there are many example of Green functors which are not Tambara functors, we see that $\mathbb{K}$ cannot possibly send all monoids with respect to the Day convolution to spectra which are stably equivalent to genuine equivariant ring spectra. 
\end{remark}

In future work, we intend to investigate which monoids will be sent to ring spectra by $\mathbb{K}$. The expectation is that this work will require us to consider how $K$-theory interacts with the ``$G$-commutative monoids'' of \cite{HillHopkins}.  In Mackey functors, work of Hoyer \cite{Hoyer}, Mazur \cite{Mazur}, and the second-named author \cite{Chan:tamb} shows that such objects recover the notion of Tambara functors.  We note that comparison of $G$-commutative monoids (also called normed algebras), Tambara functors, and equivariant ring spectra appears in recent work of Cnossen--Haugseng--Lenz--Linskens and Lenz--Linskens--P\"utst\"uck \cite{CHLL,LLP}.  Multiplicative approaches to equivariant $K$-theory which do not use spectral Mackey functors appear in work of Yau and Guillou--May--Merling--Osorno \cite{Yau:multiplicativeKtheory,GMMO-final}.
\appendix

\section{The Covariant Guillou--May Theorem}\label{app:GM thm}

This appendix contains the proof of \cref{thm:GM no op}, using the results of \cite{Schwede-Shipley} and \cite{GuillouMay:2011} to produce a zig-zag of Quillen equivalences between orthogonal $G$-spectra and covariant spectral Mackey functors.
For the readers' convenience, we gather in one place the essential points regrading the model category structure on orthogonal $G$-spectra which we will need.  
\begin{theorem}[{\cite[Theorem III.4.2, Propositions V.3.4 and V.3.8]{MM}}]
    The category $\Sp^G$ of orthogonal $G$-spectra indexed on a complete universe with the stable model structure in which it is a closed symmetric monoidal model category. The weak equivalences are called $G$-stable equivalences. Additionally, it satisfies the following properties.
    \begin{enumerate}
        \item $\Sp^G$ is compactly generated, stable, and proper.
        \item There is a strong monoidal suspension spectrum functor $\Sigma^{\infty}_G$ from pointed $G$-spaces to $\Sp^G$.  If $A$ is a finite $G$-set then $\mathbb{A} = \Sigma^{\infty}_G(A_+)$ is cofibrant in the stable model structure.
        \item The sphere spectrum $\bS = \Sigma^{\infty}_G(S^0)$ is the unit for the monoidal smash product $\wedge$ on $\Sp^G$ and is cofibrant.
        \item There is a Quillen adjunction
        \[
            \begin{tikzcd}
                \Sp \ar[bend left]{r}{\mathrm{triv}}\ar[phantom, "\perp" description, xshift=-0.5ex]{r} &  [2.5em]\ar[bend left]{l}{(-)^G}\Sp^G 
            \end{tikzcd}
        \]
        where $\Sp$ is the category of orthogonal spectra with the stable model structure.  The left adjoint $\mathrm{triv}$ is strong monoidal, hence the right adjoint $(-)^G$ is lax monoidal.
        \item We write $F(X,Y)\in \Sp^G$ for the internal mapping $G$-spectrum of two orthogonal $G$-spectra $X$ and $Y$. The category $\Sp^G$ obtains an $\Sp$-enrichment with mapping spectra between $X$ and $Y$ given by $F(X,Y)^G$.
    \end{enumerate}
\end{theorem}

In the next subsection, we recall the basic definitions of spectral $G$-categories and the methods we will use to produce our desired zig-zag of Quillen equivalences. The reader familiar with these ideas may proceed directly to \cref{subsec:pf of GM}.

\subsection{Spectral \texorpdfstring{$G$}{TEXT}-categories and \texorpdfstring{$G$}{TEXT}-quasi-equivalences}

The following definitions are adaptations of \cite[Appendix A]{Schwede-Shipley} to the context of orthogonal spectra, see also \cite{Guillou--May:enrichedModelCats} and \cite{DuggerEnr}.

\begin{definition}
A \textit{$G$-spectral category} $\cat R$ is a small category enriched in the closed symmetric monoidal category $\Sp^G$ of orthogonal $G$-spectra.
In particular, for pair of objects $r, s$ in $\cat R$, there is a morphism orthogonal $G$-spectrum $\cat R(r,s)$, together with a unit map of $G$-spectra $\bS\to \cat R(r,r)$ for all objects $r$, and an associative and unital composition map of $G$-spectra
\[
\circ\colon \cat R (s,t) \wedge \cat R(r,s)\longrightarrow \cat R (r,t)
\]
for any triple of objects $r,s$ and $t$ in $\cat R$.
\end{definition}

In particular, the data of a ring orthogonal $G$-spectrum $R$ is precisely the data of a spectral category $\cC_R$ with one object $*$ and hom $G$-spectrum $\cC_R(*,*)=R$.

\begin{definition}
    A \textit{morphism $\Phi\colon \cR\to \cO$ of $G$-spectral categories} is an $\Sp^G$-enriched functor i.e., it is a map of sets $\Phi\colon \ob(\cR)\to \ob(\cO)$ together with maps of $G$-spectra
    \[
    \Phi_{r,s}\colon \cR(r, s)\longrightarrow \cO(\Phi(r), \Phi(s))
    \]
    that preserve units and respect composition. We say $\Phi\colon \cR\to \cO$ is a \textit{stable $G$-equivalence} if $\Phi$ induces a bijection on objects, and if for all object $r, s$ in $\cR$, the map $\Phi_{r,s}$ is a stable $G$-equivalence of orthogonal spectra.
    A \textit{zig-zag of stable $G$-equivalences} between $\cR$ and $\cO$ consists of stable equivalences between spectral categories connecting $\cR$ and $\cO$:
    \[
    \begin{tikzcd}
        \cR & \cat E_1 \ar{l} \ar{r} & \cat E_2' & \cat E_2 \ar{l}\ar{r} & \cdots & \cat E_{n-1} \ar{l} \ar{r} & \cat E_{n}' &\cat E_n  \ar{l}\ar{r} & \cO.
    \end{tikzcd}
    \]
\end{definition}

\begin{definition}
    Given a $G$-spectral category $\cR$, the \emph{opposite category} $\cR^\op$ is the $G$-spectral category with the same objects as in $\cR$, with hom $G$-spectrum given by $\cR^\op(r,s)=\cR(s,r)$ for all object $r,s$ in $\cR$.
    A $G$-spectral morphism $\cR \to \cO$ determines a $G$-spectral morphism $\cR^\op\to \cO^\op$.
\end{definition}

\begin{definition}
 Let $\cR$ and $\cO$ be $G$-spectral categories. 
 Define $\cR\wedge \cO$ to be the $G$-spectral category which has objects $\ob(\cR)\times \ob(\cO)$, and hom $G$-spectra
 \[
    (\cR\wedge \cO)((r, o), (r', o'))=\cR(r,r')\wedge \cO(o,o')
 \] 
 for $r$ and $r'$ in $\cR$ and $o$ and $o'$ in $\cO$.
Spectral morphisms $\Phi\colon \cR'\to \cR$ and $\Psi\colon \cO'\to \cO$ determine a spectral morphism $\Phi\wedge \Psi \colon \cR'\wedge \cO' \to \cR\wedge \cO$.
\end{definition}

\begin{definition}
    Given $G$-spectral categories $\cat R$ and $\cat O$, an \textit{$(\cat R, \cat O)$-bimodule} is a $G$-spectrally enriched functor $\cat M\colon {\cat O}^{\op}\wedge \cat R\to \Sp^G$.
    In particular, it consists of a collection of orthogonal $G$-spectra $\cM(o, r)$ for all objects $r\in \cR$ and $o\in \cO$, together with actions:
    \[
    \cR(r,r')\wedge \cM(o,r) \wedge \cO(o', o) \longrightarrow \cM(o',r')
    \]
    for all objects $r,r'\in \cR$ and $o, o'\in \cO$, such that the actions are suitably associative and unital with respect to the spectral-enriched structures on $\cR$ and $\cO$.
    Viewing the $G$-sphere spectrum $\bS$ as a $G$-spectral category $\cC_\bS$, a $(\cR, \cC_\bS)$-bimodule shall be referred as a \textit{left $\cR$-module}, while a $(\cC_\bS, \cO)$-bimodule is referred as a \textit{right $\cO$-module}.
    Of course a $(\cR, \cO)$-bimodule is in particular a left $\cR$-module and a right $\cO$-module.
\end{definition}

\begin{remark}\label{remark: restriction of scalars for spectral categories}
    Given morphism of $G$-spectral categories $\Phi\colon \cR' \to \cR$  and $\Psi\colon \cO'\to \cO$, then any $(\cR, \cO)$-bimodule $\cM$ defines a $(\cR', \cO')$-bimodule by restriction of scalars, i.e., via precomposition:
    \[
    \begin{tikzcd}
        (\cO')^\op\wedge \cR' \ar{r}{\Psi^\op \wedge \Phi} &[1em] \cO^\op \wedge \cR \ar{r}{\cM} & \Sp^G.
    \end{tikzcd}
    \]
\end{remark}

\begin{example}\label{example: Sp^G as a bimodule}
    As $\Sp^G$ is closed monoidal, then it can be regarded as a $(\Sp^G, \Sp^G)$-bimodule via its internal hom:
    \[
    \begin{tikzcd}
        (\Sp^G)^\op\wedge \Sp^G \ar{r}{F(-,-)} & [1em] \Sp^G.
    \end{tikzcd}
    \]
    The monoidal product in $\Sp^G$ instead turns into a $(\Sp^G\wedge \Sp^G)$-left module structure, or said differently, a $(\Sp^G, (\Sp^G)^\op)$-bimodule:
    \[
    \begin{tikzcd}
         \Sp^G\wedge \Sp^G \ar{r}{-\wedge-} & \Sp^G. 
    \end{tikzcd}
    \]
\end{example}

\begin{definition}
    Let $\cR$ and $\cO$ be two $G$-spectral categories with the same set of objects $I$.
    A \emph{$G$-quasi-equivalence} between $\cR$ and $\cO$ is an $(\cR, \cO)$-bimodule $\cM$ together with morphisms $\varphi_i\colon \bS\to \cM(i,i)$ for all $i\in I$, such that for all pairs $i$ and $j$ of objects, the following composites are stable $G$-equivalences:
    \[
    \begin{tikzcd}
        \cR(i,j)\simeq \cR(i,j)\wedge \bS \ar{r}{1\wedge \varphi_i} & \cR(i,j)\wedge \cM(i,i) \ar{r} & \cM(i,j),
    \end{tikzcd}
    \]
    \[
    \begin{tikzcd}
        \cO(i,j)\simeq \bS\wedge \cO(i,j) \ar{r}{\varphi_j\wedge 1} & \cM(j,j)\wedge \cO(i,j) \ar{r} & \cM(i,j),
    \end{tikzcd}
    \]
    where the unlabeled arrows are given by the left $\cR$-action and right $\cO$-action.
\end{definition}

\begin{proposition}[{\cite[A.2.3]{Schwede-Shipley}, \cite[2.10]{Guillou--May:enrichedModelCats}}]\label{proposition: quasi equiv gives weak equiv zigzag}
Let $\cR$ and $\cO$ be two $G$-spectral categories with the same set of objects. If a $G$-quasi-equivalence exists between $\cR$ and $\cO$, then there is a zig-zag of stable $G$-equivalences between $\cR$ and $\cO$.
\end{proposition}

If $G$ is trivial, we shall refer to previous notions as spectral categories, morphisms of spectral categories, stable equivalence of spectral categories, bimodules over spectral categories and quasi-equivalence of spectral categories. 

Suppose that $\cat{V}$ and $\cat{W}$ are closed monoidal categories, and $F\colon \cat{V}\to\cat{W}$ is a lax monoidal functor.  If $\cat{C}$ is enriched in $\cat{V}$, there is a $\cat{W}$-enriched category $F_*(\cat{C})$ with same objects and hom $\cat W$-object given by
\[
    F_*(\cat{C})(x,y) = F(\cat{C}(x,y)).
\]
Composition in $F_*(\cat{C})$ comes from applying $F$ to the composition in the $\cat{V}$-category $\cat{C}$, and using the lax monoidality of $F$.  This construction is functorial in $\cat{V}$-enriched functors.  

We apply this construction in the case where $F$ is the fixed point functor $(-)^G\colon \Sp^G\to \Sp$.  Thus every $G$-spectral category determines a spectral category in a functorial way.  In nice cases, this construction sends stable $G$-equivalences to stable equivalences.

\begin{lemma}\label{lemma: change of enrichment}
    If $\cat{R}$ is a $G$-spectral category, then there is a spectral category $\Gamma(\cat{R})$ with the same objects and morphism orthogonal spectra $\Gamma(\cat{R})(r,s) = \cat{R}(r,s)^G$.  This construction is functorial in enriched functors.  Moreover, if $\Phi\colon \cat{R}\to \cat{O}$ is a stable $G$-equivalence and the morphism $G$-spectra in $\cat{R}$ and $\cat{O}$ are all fibrant then the induced functor $\Gamma(\Phi)\colon \Gamma(\cat{R})\to \Gamma(\cat{O})$ is a stable equivalence. 
\end{lemma}

\begin{proof}
    The functor $\Gamma$ is just the standard change of enrichment functor along fixed points.  If $\Phi\colon \cat{R}\to \cat{O}$ is a $G$-stable equivalence, then for all $r,s\in \cat{R}$ we have stable $G$-equivalences
    \[
        \Phi_{r,s}\colon \cat{R}(r,s)\to \cat{O}(\Phi(r),\Phi(s))
    \]
    and the induced morphism of spectral categories is obtained by applying $(-)^G$ to the map $\Phi_{r,s}$ of $G$-spectra.  Because fixed points are a right Quillen functor, it sends stable $G$-equivalences between fibrant objects to stable equivalences of spectra.  In particular, if $\cat{R}(r,s)$ and $\cat{O}(\Phi(r),\Phi(s))$ are both fibrant, then this is a stable equivalence of spectra.  Thus, if we assume that all mapping spectra in $\cat{R}$ and $\cat{O}$ are fibrant then the induced map $\Gamma(\Phi)$ is a stable equivalence.
\end{proof}

\begin{lemma}\label{lemma: change enrichment of weak equivalences zig zag}
        Suppose that $\cat{R}$ and $\cat{O}$ are two $\Sp^G$-enriched categories connected by a zig-zag of stable $G$-equivalences.  If the morphism $G$-spectra in $\cat{R}$ and $\cat{O}$ are all fibrant, then $\Gamma(\cat{R})$ and $\Gamma(\cat{O})$ are connected by a zig-zag of stable equivalences.
\end{lemma}

\begin{proof}
    By \cref{lemma: change of enrichment}, it suffices to show that $\cat{R}$ and $\cat{O}$ can be connected by a zig-zag of stable $G$-equivalences where all $G$-spectral categories involved have fibrant mapping $G$-orthogonal spectra.  To achieve this, note that stable $G$-equivalences form the weak equivalences in a model category structure on $G$-spectral categories where the fibrant objects are $G$-spectra categories with fibrant mapping $G$-spectra \cite[Theorem 2.16]{Guillou--May:enrichedModelCats}.  If we pick a functorial fibrant replacement and then apply it to the zig-zag we started with we obtain a zig-zag between $\cat{R}$ and $\cat{O}$ satisfying the necessary fibrancy condition.
\end{proof}

\subsection{Compact generators for orthogonal \texorpdfstring{$G$}{TEXT}-spectra}\label{subsec:pf of GM}
We begin by introducing some notation. Let $A$ be a finite $G$-set, and write \[
\mathbb{A}\coloneqq \Sigma_G^\infty A_+ \text{ and } D\mathbb{A} \coloneqq F(\mathbb{A}, \mathbb{S})
\] where $\mathbb{S}$ is the $G$-sphere spectrum, which is cofibrant in the stable model structure.  Pick a fibrant replacement $\widehat{\mathbb{S}}$ for $\mathbb{S}$ and write $\widehat{D}\mathbb{A}:= F(\mathbb{A},\widehat{\mathbb{S}})$. Note that $\widehat{D}\mathbb{A}$ is fibrant because $\mathbb{A}$ is cofibrant and $\widehat{\SSS}$ is fibrant. Recall from \cite[Appendix B]{HHR} that for any finite $G$-set $A$ and any $G$-spectrum $X$, the spectra $\mathbb{A}\wedge X$ and $F(\mathbb{A},X)$ are isomorphic to the indexed wedge and indexed products
\[
    \mathbb{A}\wedge X \cong \bigvee_{a\in A} X,\quad F(\mathbb{A},X) = \prod\limits_{a\in A} X.
\] 
\begin{proposition}[{\cite[Proposition B.52]{HHR}}]\label{proposition: wirthmuller iso}
    The canonical map $\mathbb{A}\wedge X\to F(A,X)$ induced by the map from coproducts to products  is a weak equivalence.  In particular, taking $X = \mathbb{S}$, there exists an equivalence $\alpha\colon \mathbb{A}\to D\mathbb{A}$.
\end{proposition}

\begin{corollary}\label{corollary: finite things preserve weak equivalences}
    If $f\colon X\to Y$ is a stable equivalence of $G$-spectra and $A$ is a finite $G$-set then $f_*\colon F(\mathbb{A},X)\to F(\mathbb{A},Y)$ is a weak equivalence.
\end{corollary}
\begin{proof}
      There is a commutative square
    \[
        \begin{tikzcd}
            \mathbb{A}\wedge X\ar[r,"1\wedge f"] \ar[d,"\sim"] & \mathbb{A}\wedge Y \ar[d,"\sim"]\\
            F(\mathbb{A},X)\ar[r,"f_*"] & F(\mathbb{A},Y)
        \end{tikzcd}
    \]
    where the vertical maps are weak equivalences of \cref{proposition: wirthmuller iso}.  Thus the map $f_*$ is a weak equivalence if and only if the top map is a weak equivalence.  Since $\mathbb{A}$ is the suspension of a finite $G$-CW complex, it is a cellular $G$-spectrum in the sense of \cite[Definition B.57]{HHR}. Hence smashing with $\mathbb{A}$ preserves weak equivalences by \cite[Proposition B.58]{HHR}.
\end{proof}

\begin{lemma}\label{lemma: alpha hat is weak equiv}
    There is a weak equivalence $\widehat{\alpha}\colon\mathbb{A}\to \widehat{D}\mathbb{A}$.
\end{lemma}
\begin{proof}
    We may assume, without loss of generality, that we have chosen a weak equivalence $w\colon \mathbb{S}\to \widehat{\mathbb{S}}$.  The map $\widehat{\alpha}$ is the composite
    \[
        \mathbb{A}\xrightarrow{\alpha} D\mathbb{A}\xrightarrow{w_*} \widehat{D}\mathbb{A}
    \]
     where $\alpha$ is  weak equivalence from \cref{proposition: wirthmuller iso}. To see that $\widehat{\alpha}$ is a weak equivalence it suffices to show that $w_*$ is a weak equivalence. This follows from the fact that $w$ is a weak equivalence and \cref{corollary: finite things preserve weak equivalences}.
\end{proof}

\begin{notation}
 Let $Q\widehat{D}\mathbb{A}$ be a cofibrant replacement of $\widehat{D}\mathbb{A}$.  We make choices of acyclic fibrations
\[
    p\colon Q\widehat{D}\mathbb{A}\xrightarrow{\sim} \widehat{D}\mathbb{A}
\]
which, in particular, implies that $Q\widehat{D}\mathbb{A}$ remains fibrant. Since $\mathbb{A}$ is cofibrant, we can find a map $\ell\colon \mathbb{A}\to Q\widehat{D}\mathbb{A}$ making the square
\[
\begin{tikzcd}
    0\ar[d] \ar[r] & Q\widehat{D}\mathbb{A} \ar[d,"p"]\\
    \mathbb{A}\ar[r,"\widehat{\alpha}"] \ar[ur,"\ell"] & \widehat{D}\mathbb{A}
\end{tikzcd}
\]
commute.  By the two-of-three property $\ell$ is necessarily a weak equivalence.   
\end{notation}

\begin{notation}
    Let $X$, $Y$, and $Z$ be any $G$-spectra.  Given any map $f\colon X\to Y$ there is an induced map
\[
    X\wedge Z\xrightarrow{f\wedge Z} Y\wedge Z
\]
which can be implemented as a map of function spectra
\[
    \mu_Z\colon F(X,Y)\to F(X\wedge Z,Y\wedge Z)
\]
which is adjoint to
\[
    F(X,Y)\wedge X\wedge Z\xrightarrow{\mathrm{ev}\wedge Z} Y\wedge Z.
\]
Similarly, there is a map
\[
    \nu_{Z}\colon F(X,Y)\to F(Z\wedge X,Z\wedge Y)
\]
which implements smashing with $Z$ on the left.
Finally, given any map $h\colon X\to Y$ we write $\lambda_h\colon \SSS\to F(X,Y)$ for the adjoint.
\end{notation}

\begin{notation}
  We write $\widehat{\mathbb{A}}$ for a fibrant replacement of $\mathbb{A}$.   We may assume we have made choices of acyclic cofibrations $w\colon\mathbb{A}\to \widehat{\mathbb{A}}$ in $\Sp^G$.  Let $\mathscr{D}\subset \Sp^G$ be the full $G$-spectral subcategory whose objects are ${\mathbb{A}}$, where $A$ runs over all finite $G$-sets.  Let $\widehat{\mathscr{D}}\subset \Sp^G$ be the full $G$-spectral subcategory whose objects are the $\widehat{\mathbb{A}}$. We later consider the full $G$-spectral subcategory of $\Sp^G$ on the objects $Q\widehat{D}\mathbb{A}$, which we denote $\mathscr{D}^{\vee}$. 
  
  \begin{center}
\begin{tabular}{||c | c||} 
 \hline
 & \\ 
Full $G$-spectral subcategory of $\Sp^G$ & Spanning objects \\ [1.5ex] 
 \hline\hline
 & \\
 $\cD$ & $\AA=\Sigma_G^\infty A_+$ for a finite $G$-set $A$ \\[1.5ex]
 \hline
  & \\
 $\widehat{\cD}$ & $\AAA$ (a fixed bifibrant replacement of $\AA$)  \\ [1.5ex]
 \hline
 & \\
 $\cD^\vee$ & $Q\widehat{D}\AA$ (a fixed bifibrant replacement of $F(\AA, \widehat{\SSS})$)  \\  [2ex]
 \hline
\end{tabular}
\end{center}
\end{notation}

\begin{lemma}\label{lemma: quasi equivalence D to D hat}
    There exists a $G$-quasi-equivalence between $\mathscr{D}$ and $\widehat{\mathscr{D}}$.
\end{lemma}

\begin{proof}
    We define a $(\widehat{\cD}, \cD)$-bimodule as the $G$-spectrally enriched functor
    \begin{align*}
        \cM\colon \cD^\op\wedge \widehat{\cD}& \longrightarrow \Sp^G\\
        (\AA, \BBB) & \longmapsto F(\AA, \BBB).
    \end{align*}
    This is simply restricting (in the sense of \cref{remark: restriction of scalars for spectral categories}) the $(\Sp^G, \Sp^G)$-bimodule structure of the internal hom of $\Sp^G$ from \cref{example: Sp^G as a bimodule} using the $G$-spectral subcategories $\widehat{\cD}\subseteq \Sp^G$ and $\cD\subseteq \Sp^G$.
    In more details, it is the bimodule $\cM(\AA, \BBB)\coloneqq F(\AA, \BBB)$ where the action:
    \[
    \widehat{\cD}(\BBB, \BBB') \wedge \cM(\AA, \BBB)\wedge \cD(\AA',\AA) \longrightarrow \cM(\AA', \BBB')
    \]
    is simply the composition:
    \[
    F(\BBB, \BBB') \wedge F(\AA, \BBB) \wedge F(\AA', \AA) \longrightarrow F(\AA', \BBB').
    \]
    For each suspension $G$-spectrum $\AA$, we define a map $\varphi_\AA\colon \bS\to \cM(\AA, \AAA)$ as the composite
    \[
    \begin{tikzcd}
        \bS \ar{r} & F(\AA, \AA) \ar{r}{w_*} & F(\AA, \AAA).
    \end{tikzcd}
    \]
    Here the unlabeled map comes from the $G$-spectrally enriched structure of $\cD$. The map $w_*$ is post-composing with the acyclic cofibration $w\colon \AA\stackrel{\simeq}\hookrightarrow \AAA$.
    Then, notice that the composite
    \[
    \begin{tikzcd}
          F(\AA', \AA) \simeq \bS \wedge F(\AA', \AA) \ar{r}{\varphi_\AA\wedge 1} & F(\AA, \AAA) \wedge F(\AA', \AA) \ar{r} & F(\AA', \AAA)   
    \end{tikzcd}
    \]
    is simply the map
    \[
    \begin{tikzcd}
        F(\AA', \AA) \ar{r}{w_*} & F(\AA', \AAA)
    \end{tikzcd}
    \]
    which is a stable $G$-equivalence as $F(\AA', -)$ preserves stable equivalences by \cref{corollary: finite things preserve weak equivalences}.
    Notice also that the composite
    \[
    \begin{tikzcd}
        F(\BBB, \BBB') \simeq F(\BBB, \BBB')\wedge \bS \ar{r}{1\wedge \varphi_\BB} & F(\BBB, \BBB') \wedge F(\BB, \BBB) \ar{r} & F(\BB, \BBB')
    \end{tikzcd}
    \]
    is simply the precomposition map
    \[
    \begin{tikzcd}
        F(\BBB, \BBB') \ar{r}{w^*} & F(\BB, \BBB')
    \end{tikzcd}
    \]
    which is a stable $G$-equivalence as the functor $F(-, \BBB')$ sends acyclic cofibrations to acyclic fibrations, since $\Sp^G$ is a closed monoidal model category and $\BBB'$ is fibrant.
\end{proof}


The following proposition provides the main technical input for our variation on the Guillou--May theorem. 

\begin{proposition}\label{proposition: quasi equiv prop}
    There is a $G$-quasi-equivalence between $\mathscr{D}^{\op}$ and $\mathscr{D}^{\vee}$.
\end{proposition}
We defer the proof to the end of the section. Combining this result with \cref{lemma: quasi equivalence D to D hat}, we obtain a zig-zag of stable $G$-equivalences between $\widehat{\cat{D}}^{\op}$ and $\cat{D}^{\vee}$. Since both of these $G$-spectral categories have fibrant mapping orthogonal $G$-spectra, we apply \cref{lemma: change enrichment of weak equivalences zig zag} to obtain the following corollary.

\begin{corollary}
    There exists a zig-zag of stable equivalences between $\Gamma(\widehat{\cat{D}}^{\op})$ and $\Gamma(\cat{D}^{\vee})$.
\end{corollary}
\begin{remark}
    We note that because taking fixed points is symmetrically lax monoidal, the change of enrichment functor $\Gamma$ from $G$-spectral categories to spectral categories commutes with taking opposite categories. In particular, $\Gamma(\widehat{\cat{D}}^{\op})\cong \Gamma(\widehat{\cat{D}})^{\op}$.
\end{remark}

We can now prove the main theorem of this appendix.

\begin{proof}[{Proof of \cref{thm:GM no op}}]\label{proof of covariant guillou-may}
    Since $\mathscr{D}^{\vee}$ contains objects weakly equivalent to the suspensions of finite $G$-sets, by \cref{lemma: alpha hat is weak equiv}, it contains a collection of compact generators for the equivariant stable homotopy category.  By the Schwede--Shipley theorem \cite[Theorem 3.3.3]{Schwede-Shipley}, there is a zig-zag of $\Sp$-enriched Quillen equivalences between $\Sp^G$ and $\Sp$-enriched presheaves on $\Gamma(\mathscr{D}^{\vee})$.  By \cref{proposition: quasi equiv prop}, the category of spectral presheaves on $\Psi(\mathscr{D}^{\vee})$ is Quillen equivalent to the category of spectral functors $\Fun_{\Sp}(\Gamma(\mathscr{\widehat{D}}),\Sp)$.  By \cite[Theorem 2.6]{GuillouMay:2011}, and \cref{lemma: change enrichment of weak equivalences zig zag} there is a zig-zag of stable equivalences of spectral categories between $K^{\GMMO}_{\bullet}(\mathcal{A}_{\PC})$ and $\Gamma(\widehat{\mathscr{D}})$ which induces a Quillen equivalence on functor categories.  Putting all this together, we have established zig-zags
    \[
        \Sp^G\leftrightarrow \Fun(\Gamma(\cat{D}^{\vee})^{\op},\Sp) \leftrightarrow \Fun(\Psi(\widehat{\cat{D}}),\Sp)\leftrightarrow \Fun(K^{\GMMO}_{\bullet}(\mathcal{A}_{\PC}),\Sp)
    \]
    of Quillen equivalences.
\end{proof}

\begin{remark}\label{remark: co-representable functors}
    We note that the string of stable equivalences relating $\Gamma(\cat{D})$ and $K_{\bullet}^{\GMMO}(\mathcal{A}^G_{\PC})$ are all the identity on objects.  Thus the spectral functor co-represented by a finite $G$-set $A$ in $K^{\GMMO}_{\bullet}(\mathcal{A}^{G}_{\PC})$ corresponds, under the string of Quillen equivalences, to the functor represented by $Q\widehat{D}\mathbb{A}\in \Gamma(\cat{D}^{\vee})$.  Thus, in the equivariant stable homotopy category, the representable functor $K^{\GMMO}_{\bullet}(\mathcal{A}^G_{\PC})(A,-)$ is isomorphic to $\mathbb{A}\simeq Q\widehat{D}\mathbb{A}$.
\end{remark}

\begin{proof}[Proof of \cref{proposition: quasi equiv prop}]
To begin, we describe the $(\mathscr{D}^{\op},\mathscr{D}^{\vee})$-bimodule, i.e.\ right $(\cD^\vee \wedge \cD)$-module, which implements the $G$-quasi-equivalence between these two $G$-spectrally enriched categories.  Define the bimodule as 
\begin{align*}
\cM\colon (\cD^\vee)^\op \wedge \cD^\op &\longrightarrow \Sp^G\\
(Q\widehat{D}\AA, \BB) & \longmapsto F(Q\widehat{D}\mathbb{A}\wedge \mathbb{B},\widehat{\SSS}).
\end{align*}
Recall from \cref{example: Sp^G as a bimodule} that $\Sp^G$ is a bimodule via its internal hom. In particular, for any orthogonal $G$-spectrum $X$, the functor $F(-,X)\colon (\Sp^G)^\op\to \Sp^G$ is a morphism of $G$-spectral categories. Applying this to $X=\widehat{\SSS}$, we obtain a $((\Sp^G)^\op, \Sp^G)$-bimodule
\[
\begin{tikzcd}
    (\Sp^G)^\op\wedge (\Sp^G)^\op \ar{r}{-\wedge -} & [1em](\Sp^G)^\op \ar{r}{F(-, \widehat{\SSS})} & [1em]\Sp^G
\end{tikzcd}
\]
by using the left $(\Sp^G\wedge \Sp^G)$-module structure given by the monoidal product in $\Sp^G$ from \cref{example: Sp^G as a bimodule}.
Then by restricting onto the $G$-spectral full subcategories $\cD$ and $\cD^\vee$, in the sense of \cref{remark: restriction of scalars for spectral categories}, we obtain the desired $(\mathscr{D}^{\op},\mathscr{D}^{\vee})$-bimodule $\cM$ above.
More explicitly, the bimodule structure maps can be described as follows.  For finite $G$-sets $A$, $B$, and $C$ the composite 
\begin{align*}
    \cM(B,C)\wedge \mathscr{D}^{\vee}(A,B) & = F(Q\widehat{D}\mathbb{B}\wedge \mathbb{C},\widehat{\SSS})\wedge F(Q\widehat{D}\mathbb{A}, Q\widehat{D}\mathbb{B})\\
    & \xrightarrow{\cong} F(Q\widehat{D}\mathbb{B},\widehat{D}\mathbb{C})\wedge F(Q\widehat{D}\mathbb{A}, Q\widehat{D}\mathbb{B})\\
    & \xrightarrow{\circ} F(Q\widehat{D}\mathbb{A}\widehat{D}\mathbb{C}) \cong \cM(A,C)
\end{align*}
serves as the right action map for the bimodule while
\begin{align*}
    \mathscr{D}^{\op}(B,C)\wedge \cM(A,B) &  = F(\mathbb{C},\mathbb{B})\wedge F(Q\widehat{D}\mathbb{A}\wedge \mathbb{B},\widehat{\SSS})\\
    & \xrightarrow{\nu_{Q\widehat{D}\mathbb{A}}\wedge 1} F(Q\widehat{D}\mathbb{A}\wedge \mathbb{C},Q\widehat{D}\mathbb{A}\wedge\mathbb{B})\wedge F(Q\widehat{D}\mathbb{A}\wedge \mathbb{B},\widehat{\SSS})\\
    &\xrightarrow{\cong }F(Q\widehat{D}\mathbb{A}\wedge \mathbb{B},\widehat{\SSS})\wedge F(Q\widehat{D}\mathbb{A}\wedge \mathbb{C},Q\widehat{D}\mathbb{A}\wedge\mathbb{B}) \\
    & \xrightarrow{\circ} F(Q\widehat{D}\mathbb{A}\wedge \mathbb{C},\widehat{\SSS}) = \cM(A,C)
\end{align*}
serves as the left action map.  

We immediately remark that the fundamental reason this is a quasi-equivalence is because of the zig-zag of weak equivalences
\[
    \mathscr{D}^{\op}(B,A) = F(\mathbb{A},\mathbb{B})\xleftarrow{\ell_A^*} F(Q\widehat{D}\mathbb{A},\mathbb{B})\xrightarrow{({\ell_B}_*)}F(Q\widehat{D}\mathbb{A},Q\widehat{D}\mathbb{B}) = \cat{D}^{\vee}(A,B).
\]

For any finite $G$-set $A$, we let the unit map $\bS\to \cM(A,A)$ be the map
\[
    \phi_A\colon \bS\to F(Q\widehat{D}\mathbb{A}\wedge \mathbb{A},\widehat{\bS})\cong F(Q\widehat{D}\mathbb{A},\widehat{D}\mathbb{A})
\]
which picks out the adjunct of the acyclic fibration $p\colon Q\widehat{D}\mathbb{A}\to \widehat{D}\mathbb{A}$.  For any finite $G$-set $B$, the composite
\[
    \bS\wedge \cD^{\vee}(B,A)\to \cat{M}(A,A)\wedge \cD^{\vee}(B,A)\to \cat{M}(B,A)
\]
is the map
\[
    F(Q\widehat{D}\mathbb{B},Q\widehat{D}\mathbb{A})\to F(Q\widehat{D}\mathbb{A},\widehat{D}\mathbb{A})
\]
is just post composition with the weak equivalence $p\colon Q\widehat{D}\mathbb{A}\to \widehat{D}\mathbb{A}$.  Since $Q\widehat{D}\mathbb{A}$ is cofibrant and the source and target of $p$ are fibrant this is a weak equivalence.  

All that is left is to check that the composite
\begin{equation}\label{equation: right unit equivalence.}
    \cD^{\op}(A,B)\wedge \bS\xrightarrow{1\wedge \phi_A} \cD^{\op}(A,B)\wedge \cat{M}(A,A)\to \cat{M}(A,B)
\end{equation}
is a weak equivalence for any finite $G$-sets $A$ and $B$.  This is more involved than the previous verification and requires chasing a rather large diagram.  For brevity, we abbreviate $Q\widehat{D}\mathbb{A}$ to simply $Q$.  Consider the following diagram:
\[
\begin{tikzpicture}[baseline= (a).base]
\node[scale=0.90] (a) at (1,1){
\begin{tikzcd}
	{F(\mathbb{B},\mathbb{A})\wedge \SSS} && { F(\mathbb{B},\mathbb{A})\wedge F(Q \wedge \mathbb{A},\widehat{\SSS})} & { F(Q\wedge \mathbb{B},Q\wedge \mathbb{A})\wedge F(Q\wedge \mathbb{A},\widehat{\SSS})} & {F(Q\wedge \mathbb{B},\widehat{\SSS})} \\
	{F(\mathbb{B},\mathbb{A})} &&& {F(\mathbb{A}\wedge \mathbb{B},Q\wedge \mathbb{A})\wedge F(Q\wedge \mathbb{A},\widehat{\SSS})} \\
	&&& {F(\mathbb{A}\wedge \mathbb{B},\mathbb{A}\wedge \mathbb{A})\wedge F(Q\wedge \mathbb{A},\widehat{\SSS})} & {F(\mathbb{A}\wedge \mathbb{B},\widehat{\SSS})} \\
	&& { F(\mathbb{B},\mathbb{A})\wedge F(\mathbb{A}\wedge \mathbb{A},\widehat{\SSS})} & { F(\mathbb{A}\wedge \mathbb{B},\mathbb{A}\wedge \mathbb{A})\wedge F(\mathbb{A}\wedge \mathbb{A},\widehat{\SSS})} \\
	{F(\mathbb{B},\widehat{D}\mathbb{A})} &&&& {F(\mathbb{B}\wedge \mathbb{A},\widehat{\SSS})}
	\arrow[""{name=0, anchor=center, inner sep=0}, "{1\wedge \phi_{A}}", from=1-1, to=1-3]
	\arrow["\cong"', from=1-1, to=2-1]
	\arrow["{1\wedge\lambda}"', from=1-1, to=4-3]
	\arrow["{\nu_Q\wedge 1}", from=1-3, to=1-4]
	\arrow["{\mu_A\wedge 1}"', curve={height=12pt}, from=1-3, to=3-4]
	\arrow["{1\wedge (\ell\wedge 1)^*}"', from=1-3, to=4-3]
	\arrow["{\circ'}", from=1-4, to=1-5]
	\arrow["{(\ell\wedge 1)^*\wedge 1}", from=1-4, to=2-4]
	\arrow["{\sim (\ell\wedge 1)^*}", from=1-5, to=3-5]
	\arrow["{\widehat{\alpha}_* \sim}"', from=2-1, to=5-1]
	\arrow["{\circ'}", from=2-4, to=3-5]
	\arrow["{(\ell\wedge 1)_*\wedge 1}"', from=3-4, to=2-4]
	\arrow["{1\wedge(\ell\wedge 1)^*}", from=3-4, to=4-4]
	\arrow["\cong", from=3-5, to=5-5]
	\arrow[""{name=1, anchor=center, inner sep=0}, "{\mu_A\wedge 1}", from=4-3, to=4-4]
	\arrow["{\circ'}"', from=4-4, to=3-5]
	\arrow[""{name=2, anchor=center, inner sep=0}, "\cong"', from=5-1, to=5-5]
	\arrow["{(1)}"{pos=0.3}, draw=none, from=0, to=4-3]
	\arrow["{(2)}", draw=none, from=1, to=2]
\end{tikzcd}
};  
\end{tikzpicture}
\]
where $\lambda\colon \SSS\to F(\mathbb{A}\wedge \mathbb{A},\widehat{\SSS})$ is the map which picks outs the composite
\[
    \mathbb{A}\wedge \mathbb{A}\xrightarrow{\widehat{\alpha}\wedge 1} \widehat{D}\mathbb{A}\wedge A\xrightarrow{\mathrm{ev}} \widehat{\SSS}
\]
and $\circ'$ is the composition precomposed with the symmetry isomorphism.
Since $\phi_A$ is picking out the composite
\[
    Q\wedge \mathbb{A}\xrightarrow{p\wedge 1} \widehat{D}\mathbb{A}\xrightarrow{\mathrm{ev}} \widehat{\SSS}
\]
the triangle (1) commutes because $\widehat{\alpha} = p\circ \ell$, by definition of $\widehat{\alpha}$.

The region (2) commutes because either way of going around computes the adjoint of a map obtain by post composition with $\widehat{\alpha}\colon \AA\to \widehat{D}\AA$.  The remaining regions in the diagram all commute as a result of standard diagram chases in a closed monoidal category.  The left vertical map $\widehat{\alpha}_*$ is a weak equivalence because $\widehat{\alpha}$ is a weak equivalence and $B$ is a finite $G$-CW complex.  The right vertical map $(\ell\wedge 1)^*$ is an equivalence because $\ell$ is a weak equivalence, $\widehat{\SSS}$ is fibrant, and $Q$, $A$, and $B$ are all cofibrant.  Thus the composite along the top is a weak equivalence.  Since this is precisely \cref{equation: right unit equivalence.} the proof is complete.
\end{proof}
\bibliographystyle{alpha}
\bibliography{sample}
\end{document}